\documentclass[twoside,11pt]{article}
\usepackage{jmlr2e}

\usepackage{setspace}
\usepackage{graphicx}
\usepackage{verbatim}
\usepackage{appendix}
\usepackage{amsmath}
\usepackage{tablefootnote}

\usepackage{natbib} 

\usepackage[colorlinks=false]{hyperref}

\usepackage{graphicx}
\usepackage{subcaption}
\usepackage{amsmath}
\usepackage{amssymb}
\usepackage{tikz}

\newcommand{\julcom}[1]{\textcolor{magenta} {\emph{Julien says: #1}}}

\newcommand{\matnorm}[1]{{\left\vert\kern-0.25ex\left\vert\kern-0.25ex\left\vert #1 
    \right\vert\kern-0.25ex\right\vert\kern-0.25ex\right\vert}}
\newcommand{\opnorm}[1]{{\left\vert\kern-0.25ex\left\vert\kern-0.25ex\left\vert #1 
    \right\vert\kern-0.25ex\right\vert\kern-0.25ex\right\vert}}		

\newcommand{\norm}[1]{\left\Vert#1\right\Vert}

\newcommand{\Real}{\mathbb R}

\newcommand{\nat}{\mathbb N}

\newcommand{\argmin}{\text{argmin}}
\newcommand{\argmax}{\text{argmax}}

\newcommand{\floor}[1]{ \left\lfloor #1 \right\rfloor }
\newcommand{\ceil}[1]{ \left\lceil #1 \right\rceil }

\newcommand{\data}{\ensuremath{ \mathcal D} }

\newcommand{\inspace}{\ensuremath{ \mathcal X}}
\newcommand{\outspace}{\ensuremath{ \mathcal Y}}

\newcommand{\convto}{\longrightarrow}

\newcommand{\bd}{\leadsto} 
\newcommand{\bdu}{\twoheadrightarrow} 

\newcommand{\grid}{\ensuremath{  G}}

\newcommand{\metric}{\, \mathfrak{d}} 

\newcommand{\predfn}{\, \mathfrak{  \hat f_n}} 

\newcommand{\hestthresh}{\ensuremath{ \lambda}}

\newcommand{\obserrbnd}{\bar{\mathfrak e}}

\newcommand{\decke}{\ensuremath{\mathfrak u}}
\newcommand{\boden}{\ensuremath{\mathfrak l}}

\newcommand{\seq}[2]{\ensuremath{\Bigl(#1\Bigr)_{#2}}}

\renewcommand{\Pr}{\mathrm{Pr}}

\newcommand{\pmeas}{\ensuremath{\mathbb P}} 

\newcommand{\beq}{\begin{equation}}
\newcommand{\eeq}{\end{equation}}

\newcommand{\tinc}{\ensuremath{ \Delta}}

\allowdisplaybreaks
\begin{document}




\title{Lipschitz Interpolation: Non-parametric Convergence under Bounded Stochastic Noise}

\author{\name Julien Walden Huang \email julien.huang@sjc.ox.ac.uk \\
      \addr 
      University of Oxford
      \AND
      \name Stephen Roberts \email sjrob@robots.ox.ac.uk \\
      \addr
      University of Oxford
      \AND
      \name Jan-Peter Calliess \email jan-peter.calliess@oxford-man.ox.ac.uk\\
      \addr 
      University of Oxford}




\maketitle 

\begin{abstract}                         
This paper examines the asymptotic convergence properties of Lipschitz interpolation methods within the context of bounded stochastic noise. In the first part of the paper, we establish probabilistic consistency guarantees of the classical approach in a general setting and derive upper bounds on the uniform convergence rates. These bounds align with well-established optimal rates of non-parametric regression obtained in related settings and provide new precise upper bounds on the non-parametric regression problem under bounded noise assumptions. Practically, they
can serve as a theoretical tool for comparing Lipschitz interpolation to alternative non-parametric regression methods, providing a condition on the behaviour of the noise at the boundary of its support which indicates when Lipschitz interpolation should be expected to asymptotically outperform or underperform other approaches.
In the second part, we expand upon these results to include asymptotic guarantees for online learning of dynamics in discrete-time stochastic systems and illustrate their utility in deriving closed-loop stability guarantees of a simple controller. We also explore applications where the main assumption of prior knowledge of the Lipschitz constant is removed by adopting the LACKI framework (\cite{calliess2020lazily}) and deriving general asymptotic consistency. 
\end{abstract}




\section{Introduction} \label{sec:intro}

Non-parametric regression methods are a flexible class of machine learning algorithms that often enjoy strong estimation success even when little is known about the underlying target function or data-generating distribution. Generally,  theoretical guarantees on their convergence have been well-researched with classical optimal convergence rates obtained in seminal work by \cite{stone1982optimal} in the case where the noise is assumed to be essentially Gaussian; although various extensions relax this assumption, and a variety of general consistency results can be derived (\cite{gyorfi2002distribution}). In this paper, we will extend the literature on the non-parametric estimation problem by considering the setting of bounded stochastic noise and  a non-parametric regression framework known as Lipschitz interpolation\footnote{Also referred to as non-linear set membership (\cite{milanese2004set}) or kinky inference (\cite{calliess2020lazily})} (\cite{beliakov2006interpolation}. This theoretical context is of particular interest to the field of control where Lipschitz interpolation has been used in various predictive control applications and bounded  noise assumptions are commonly made.

More generally, the popularity of data-driven adaptive control frameworks has increased significantly in both industry and academia over the past two decades. These frameworks assume that the underlying system dynamics are partially or completely unknown and need to be identified through computational machine learning-based approaches. Traditionally, linear parametric regression methods have been extensively studied and utilized for this purpose (\cite{ljung2010perspectives}). However, with recent advancements, the research focus has expanded to encompass non-linear approaches, reflecting the growing interest in more expressive modelling techniques (\cite{schoukens2019nonlinear}). Non-linear non-parametric regression methods such as the Nadaraya-Watson estimator (\cite{nadaraya1964estimating}), Polynomial Spline Interpolation (\cite{stone1994use}), general Lipschitz Interpolation or Gaussian process regression (\cite{williams2006gaussian}) which offer flexible predictors capable of modelling complex dynamics with minimal hyperparameter tuning have emerged as natural tools in this context.
In particular, Gaussian process regression and Lipschitz interpolation frameworks have been increasingly studied due to their inherent ability to provide uncertainty quantification and worst-case error guarantees which enable the design of robust controllers that ensure system safety and meet desired performance criterias (e.g. see \cite{hewing2019cautious}, \cite{canale2007power}).

In contrast to the probabilistic nature of the uncertainty characterisation provided by Gaussian processes, Lipschitz interpolation-based techniques offer deterministic guarantees in the form of feasible systems sets (\cite{milanese2004set}) and worst-case error bounds (\cite{calliess2020lazily}). This has been particularly useful in the context of safety-critical model predictive control (MPC) where a number of associated non-linear controllers have been designed and are being researched (see \cite{canale2014nonlinear}, \cite{manzano2020robust}, \cite{manzano2021componentwise} for a selection). At the same time, this  popularity has led to several recent extensions of the original Lipschitz interpolation framework: (\cite{calliess2020lazily}) relaxes the assumption of prior knowledge of the Lipschitz constant in favour of a fully data-driven approach by incorporating a Lipschitz constant estimation procedure, (\citet{maddalena2020learning}) proposes an equivalent smooth formulation which is more suited for controllers that rely on gradient computations,
(\cite{blaas2019localised}) extends the framework by incorporating localised Lipschitz constants and (\cite{manzano2022input}) proposes a computationally more efficient approach that retains key properties of the original Lipschitz interpolation framework.

Given the growing use of Lipschitz interpolation frameworks in control, obtaining a strong theoretical understanding of this method is essential. While several finite sample guarantees and worst-case error bounds already exist (see in particular \cite{milanese2004set}, \cite{calliess2020lazily}), few consistency results have been derived, to the best of our knowledge, none under stochastic noise. 
By contrast, numerous asymptotic guarantees and convergence rates have been obtained for other popular non-parametric methods. In particular, for alternative safe-learning frameworks based on Gaussian processes, both the pointwise convergence of the posterior mean function (\cite{seeger2008information}, \cite{yang2017frequentist}, \cite{wynne2021convergence}) and the contraction rate of the posterior distribution, which provide a measure of uncertainty quantification (\cite{van2008rates}, \cite{van2011information}), of the fitted Gaussian processes have been derived. 

These types of asymptotic properties are crucial for adaptive control applications as they guarantee that the learned dynamics and error bounds accurately converge to the true underlying system dynamics while also providing a characterisation of the long-run performance of the regression method. This in turn ensures that the controllers built on these data-driven frameworks become increasingly more successful the longer the interaction with the underlying plant progresses. Considering the computational advantages of Lipschitz interpolation over Gaussian process regression (\cite{calliess2020lazily}),  deriving analogous asymptotic guarantees for Lipschitz interpolation is therefore strongly desirable and constitutes the main motivation of this paper. Specifically, the following contributions to the literature are made:

\begin{itemize}
    \item \textbf{(Main Result)} In the case of independent input sampling, general consistency and upper bounds on the asymptotic convergence rates are obtained for both the prediction function (Theorem \ref{thm:conv rate}) and the worst-case error bounds (Corollary \ref{cor:FSS}) of the general Lipschitz constant interpolation framework. While convergence lower bounds do not exist for the exact setting considered in this paper signifying  that the optimality of our bounds is not (yet) established, the obtained rates are consistent with the optimal convergence rates for non-parametric regression in related settings; e.g. with the classical convergence rate results derived by (\cite{stone1982optimal}). 
    \item In the case of discrete-time non-linear and noisy dynamical systems, we show that the Lipschitz interpolation framework and worst-case bounds converge point-wise in moments (Corollary \ref{cor:online} and ensuing discussion) and that, under an additional sampling assumption, the convergence rates match the ones derived in the first part of the paper.
    The first result can be directly applied in the context of the existing non-linear controllers discussed above (e.g. \cite{canale2014nonlinear}, \cite{manzano2020robust}) and we provide a simple illustration in the context of online learning-based trajectory tracking control (see Section \ref{sec: online control}).
\item In a general sampling setting, probabilistic consistency is shown (Theorem \ref{thm:lacki}) for the  fully data-driven LACKI (\textit{Lazily Adapted Constant Kinky Inference}) estimator (\citet{calliess2020lazily}) that extends the general Lipschitz interpolation framework by removing the key assumption of prior knowledge of the Lipschitz constant. This result generalises on Theorem 16 of (\cite{calliess2020lazily}) which derives the consistency of the LACKI estimator in the noise-free setting.
\end{itemize}

We note that with the goal of obtaining a precise characterisation of the convergence rates of Lipschitz interpolation methods, we make a non-standard noise assumption (Assumption \ref{ass:noise lip}) utilising the concept of "non-regular" noise (\cite{ibragimov2013statistical}) which describes the behaviour of the tails of the noise distribution in proximity of assumed noise bounds. This type of assumption has been used in recent research on non-parametric boundary regression (see \citet{hall2009nonparametric}, \cite{jirak2014adaptive} and ensuing works) and allows for a better comparison between the convergence rates of Lipschitz interpolation derived in this paper and the ones known for Gaussian process regression and other kernel methods. In fact, the convergence rate bounds obtained in this paper provide an explicit condition on the tail behaviour of the noise that indicates when the Lipschitz interpolation should be expected to asymptotically outperform or underperform
other non-parametric regression paradigms.

\section{Lipschitz Interpolation: Set-up \& Assumptions}

Given an input space $\inspace \subset \Real^d$ endowed with a metric $\metric: \inspace^2 \to \Real_{\geq 0}$ and an output space\footnote{Here, it would be possible to extend the analysis done in this paper to a vector-valued output space, i.e. $\outspace \subset \Real^m$ for $m \in \nat$, by applying the obtained results in a component-wise fashion.} $\outspace \subset \Real$ endowed with a metric $\metric_\outspace: \outspace^2 \to \Real_{\geq 0}$, the goal of non-parametric regression is to learn an unknown target function $f:\inspace  \to \outspace$. In this paper, we will assume  that $f$ belongs to the class of $L$-Lipschitz (continuous) functions (with respect to $\metric_\inspace$, $\metric_\outspace$ and some $L\in\Real_+$) which we formally define as:
\begin{equation*}
Lip(L,\metric) :=\{ h: \inspace \to \outspace | \metric_{\outspace}(h(x), h(x')) \leq L \metric(x,x'), \forall x,x' \in \inspace\}.
\end{equation*} 
The smallest non-negative number $L^*$ for which $f$ is $L^*$-Lipschitz is called the \emph{best} Lipschitz constant of $f$, i.e. $L^* = \min\{L \in \Real_{\geq 0} | f \in Lip(L,\metric)\}$.  The functional assumption that the target function $f$ is $L$-Lipschitz continuous for some $L \in \Real_+$ is essential for the application of Lipschitz interpolation frameworks and is standard in the relevant literature (\cite{milanese2004set}, \cite{beliakov2006interpolation}, \cite{calliess2020lazily}). 

In order to learn $f$, we assume that a sequence of sample sets $(\data_n)_{n \in \mathbb{N}} := (\grid^\inspace_n, G^\outspace_n)_{n \in \mathbb{N}}$ defined such that $\data_n \subset \data_{n+1}$ for $n\in\nat$ is available, where $\grid_n^\inspace := \{s_i | i=1,...,N_n\}\subset \inspace$ represents a set of sample inputs that can be either deterministically or randomly generated and $G^\outspace_n := \{\tilde f_i | i=1,...,N_n\} \subset \outspace$ denotes the set of noise-corrupted values of the target function $f$ associated with the inputs in $\grid_n^\inspace$. Unless stated otherwise, we will also assume that elements of $G^\outspace_n$ are of the form $\tilde f_k= f(s_k)+ e_k$ where $(e_k)_{k \in \nat}$ is a collection of random variables denoting the additive observational noise.

In this paper, we will make the following assumption on the noise:

\begin{ass} \label{ass:noise lip simple}(General noise assumptions) 
The noise variables $(e_k)_{k \in \nat}$ are assumed to be independent and identically\footnote{The identically distributed assumption is made to alleviate notation and is not technically needed in our derivations.} distributed random variables with compact support:
$\exists \obserrbnd>0$ such that $\forall k\in\nat:$ $\pmeas\left(e_k \in [-\obserrbnd, \obserrbnd]\right) =1$. Furthermore, we assume that the bounds of the support are tight in the following sense:
 \newline  There exists $\bar \epsilon > 0$, $\forall k \in \nat,  \epsilon \in (0, \bar \epsilon)$:
    $$\pmeas(e_k > \obserrbnd - \epsilon) > 0 \quad \text{ and } \quad \pmeas(e_k < -\obserrbnd + \epsilon) > 0$$
\end{ass}

In order to derive precise upper bounds on the convergence rates, we will sometimes make an additional noise assumption which describes the behaviour of the noise at the boundary of its support. This assumption is given formally as follows:

\begin{ass} \label{ass:noise lip}(Assumptions on the boundary behaviour of the noise) 
Assume that Assumption \ref{ass:noise lip} holds. We assume that the behaviour of the noise near the bounds of the support can be characterised in the following sense: There exists $\bar \epsilon, \gamma, \eta > 0,\forall k \in \nat,  \epsilon \in (0, \bar \epsilon)$:
    $$\pmeas(e_k > \obserrbnd - \epsilon) > \gamma \epsilon^\eta \quad \text{ and } \quad \pmeas(e_k < -\obserrbnd + \epsilon) > \gamma \epsilon^\eta.$$
\end{ass}

\begin{figure*}[t]
\centering
 \begin{subfigure}[t]{0.25\textwidth}%
\includegraphics[width=\textwidth]{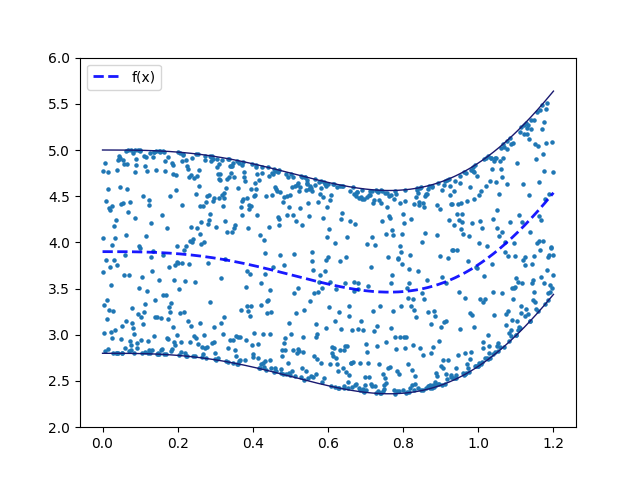}  \caption{$\eta = 0.5$}%
  \end{subfigure}%
    \begin{subfigure}[t]{0.25\textwidth}%
\includegraphics[width=\textwidth]{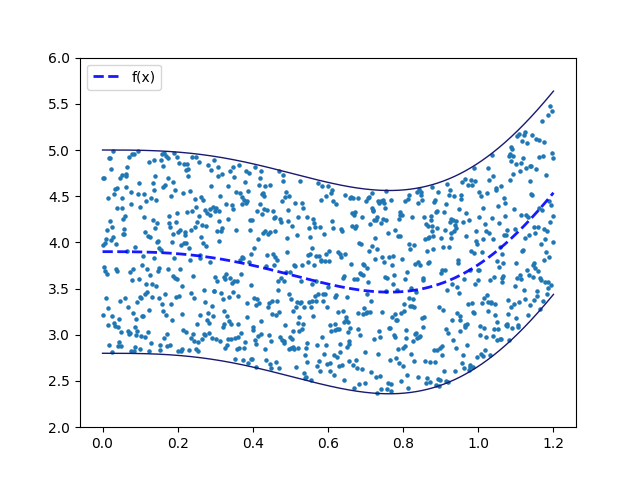} \caption{$\eta = 1$}%
  \end{subfigure}%
 \begin{subfigure}[t]{0.25\textwidth}%
	\includegraphics[width=\textwidth]{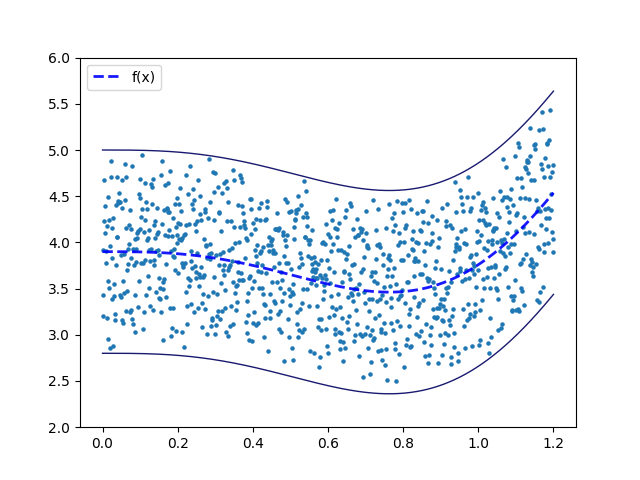}%
  \caption{$\eta = 2$}%
  \end{subfigure}%
    \begin{subfigure}[t]{0.25\textwidth}%
		\includegraphics[width=\textwidth]{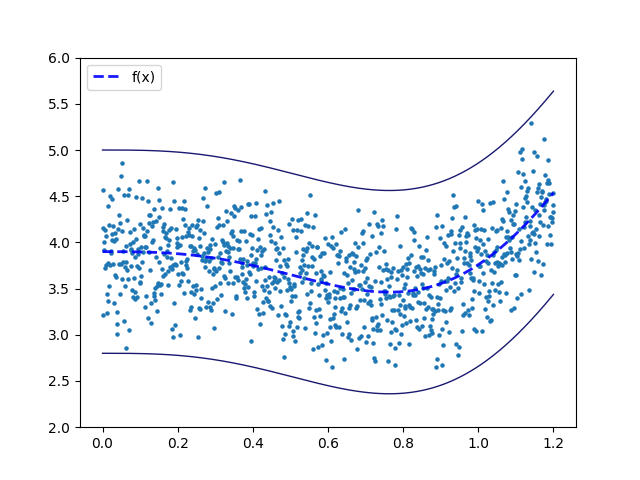}
  \caption{$\eta = 3$}%
  \end{subfigure}%
	\caption[Illustration of the behaviour of the noise variables described by Assumption \ref{ass:noise lip} for various values of $\eta$. ]{ Illustration of Assumption \ref{ass:noise lip} for various $\eta$. The target function is given by $f(x) = -\sin(3x)x^{2} + 5$ and the noise distribution is defined as a mixture of two truncated Weibull distributions. The solid lines define the error bounds of the observed data, i.e. (with abuse of notation) $f \pm \obserrbnd$.
	}	\label{Fig: lcls best}%
\end{figure*}
\begin{ex}(Noise distributions)\label{example:noise} For $\eta = 1$, commonly used noise distributions with bounded support such as the uniform or the truncated Gaussian distributions satisfy Assumption \ref{ass:noise lip}. More generally, any noise distribution for which the density can be bounded away from zero on a bounded symmetric support satisfies the assumption with $\eta = 1$.
\end{ex}
The assumption of boundedness of the support of the noise distribution given in Assumption \ref{ass:noise lip simple} is standard in the Lipschitz interpolation literature (e.g. see \cite{milanese2004set}, \cite{calliess2020lazily}) as it ensures that the functions $\decke_n, \boden_n$ defined in Definition \ref{def:KILsimplified} are generally well-behaved. By contrast, as noted in the introduction of this paper, the assumption on the tail of the noise distribution stated in Assumption \ref{ass:noise lip} is non-standard\footnote{But as shown in Example \ref{example:noise}, many standard noise distributions arise as special cases.} in the literature. While this assumption will be not needed to ensure the asymptotic consistency of Lipschitz interpolation frameworks, the precise characterisation of the bounded tail of the noise distribution as a function of $\gamma$ and $\eta$ given in Assumption \ref{ass:noise lip} makes it possible to derive a more refined convergence rate result that depends on $\eta$. 

The general Lipschitz interpolation framework considered in this paper is defined as follows:
\begin{defn}(Lipschitz interpolation) \label{def:KILsimplified}
    Using the set-up defined above, we define the sequence of predictors $( \hat f_n)_{n \in \nat}$, $\hat f_n: \inspace \to \outspace $  associated to $(\data_n)_{n \in \nat}$, as
   	\begin{equation}
	\predfn\bigl(x) := \frac{1}{2} \decke_n(x) + \frac{1}{2}  \boden_n(x), \label{eq:KIpred_basic}
	\end{equation}
	where $\decke_n, \boden_n : \inspace \to \outspace $ are defined as
	\begin{align*}
	    &\decke_n(x) = \min_{i=1,...,N_n}\tilde f_i + L \metric(x,s_i)  \\ 
	    &\boden_n(x) = \max_{i=1,...,N_n}\tilde f_i - L \metric(x,s_i)  
     \end{align*}
 and $ L \in \Real_{\geq 0}$ is a selected hyper-parameter. 
\end{defn}
Ideally, the hyper-parameter $L \in \Real_{\geq 0}$ can be set to be larger than the best Lipschitz constant $L^*$ of the unknown target function. In this case, existing finite sample and worst-case guarantees can be utilised. 

\begin{rem}\label{rem:alt}(Alternative Formulation) In some works (see in particular \cite{milanese2004set}, \cite{calliess2020lazily}) an alternative formulation is given for the Lispchitz interpolation predictors. In this case, the bounds ($\obserrbnd$) on the noise distribution are assumed known and are explicitly used in the formulation of $\tilde \decke_n, \tilde \boden_n : \inspace \to \outspace \nonumber$:
\begin{align*}
	    & \tilde \decke_n (x) = \min_{i=1,...,N_n}\tilde f_i + L \metric(x,s_i) + \obserrbnd \color{black} \\ 
	    & \tilde \boden_n(x) = \max_{i=1,...,N_n}\tilde f_i - L \metric(x,s_i) - \obserrbnd \color{black}. 
\end{align*}
This formulation is useful for computing tight worst-case upper and lower bound guarantees in practice and can be used in the context of this paper to weaken Assumption \ref{ass:noise lip simple} by considering asymmetric noise bounds, i.e. $ e \in [\obserrbnd_1, \obserrbnd_2]$ with probability $1$ where $\obserrbnd_1< 0< \obserrbnd_2 \in \Real$. All results derived in this paper can be shown to hold for the alternative Lipschitz interpolation formulation. 
\end{rem}

\begin{figure*}[t]
\centering
\includegraphics[width=0.49\textwidth]{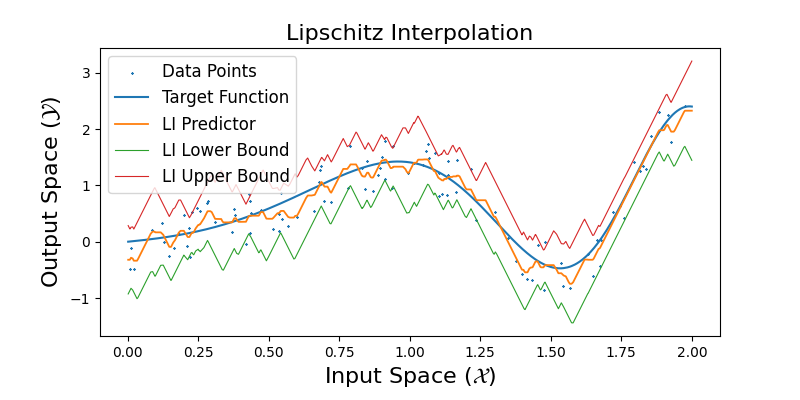}  
\includegraphics[width=0.49\textwidth]{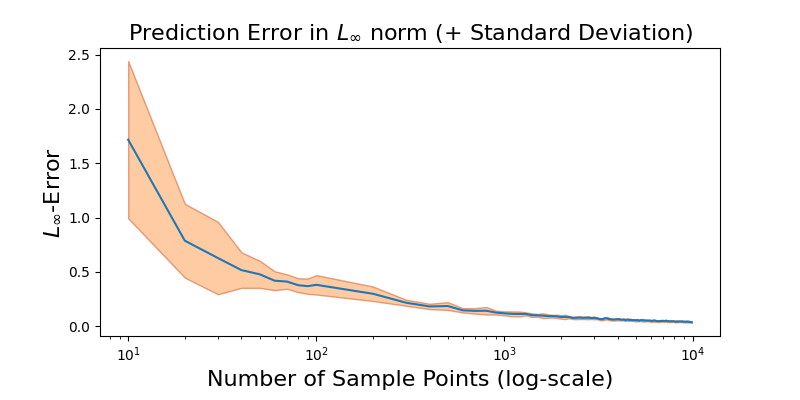}  
	\caption[Illustration of the consistency of Lipschitz interpolation]{ Illustration of 
 of the consistency of Lipschitz interpolation for the target function: $f(x) = \sqrt{x} \sin(2x^2) + 0.5x$ on the input space $\inspace = [0,2]$, with uniform sampling on $\inspace$ and with independent uniform noise: $U([-0.5,0.5])$ on the observations ($\eta=1$). The Lipschitz interpolation plotted in the lefthand figure utilised 100 samples and assumed that a bound of ($\obserrbnd'=0.7$) on the noise bound ($\obserrbnd=0.5$) was known in order to compute the lower and upper bounds (the LI predictors can be constructed without knowledge of $\obserrbnd'$). The convergence rate and standard deviation plotted on the righthand figure were obtained by running the experiment independently 20 times. Both plots assumed that access to a bound on the best Lipschitz constant was known in order to apply the Lipschitz interpolation framework.	\label{Fig: LI simple}}%
\end{figure*}

As the description of the input and output metrics has been general so far, we make the following simplifying assumption on the output metric in order to obtain our theoretical results.
\begin{ass} \label{ass:outspace norm}(Assumption on $\metric_\outspace$).
In this paper we will restrict ourselves to the case, $\metric_\outspace(y,y') = \norm{y-y'}_\outspace$, $\forall y, y' \in \mathcal{Y}$ where $\norm{.}_\outspace$ is a norm on $\outspace$. It will therefore be sufficient to derive our asymptotic results for the case: $\norm{.}_\outspace =|.|$ as discussed below.
\end{ass}

As the norms on $\outspace\subset \Real$ are of the form $\norm{y-y'} = c|y-y'|$ $\forall y, y' \in \mathcal{Y}$ for some $c>0$, it is sufficient to consider the case $\norm{y-y'}_\outspace = |y-y'|$, $\forall y, y' \in \mathcal{Y}$ in order to achieve our theoretical results. Assumption \ref{ass:outspace norm} is necessary in order to ensure that for arbitrary $ x,x' \in \inspace$, the relations: $f(x) \leq f(x') + \frac{L}{c}\metric(x,x')$ and $f(x) - \frac{L}{c}\metric(x,x') \leq f(x')$ hold. In particular, for any sub-linear metric $\metric_\outspace$, these inequalities no longer hold. We note however that no restrictions are made on the input metric $\metric$. 

\section{Asymptotic Consistency and Convergence Rates}

In order for our consistency results to hold for both random and deterministic sampling approaches, we recall  \textit{Definition 8}: \textit{"Becoming dense, rates, $\stackrel{r}{\convto}, \stackrel{r}{\bd}, \stackrel{r}{\bdu}$"} of (\cite{calliess2020lazily}) to define general sampling conditions for $(G_n)_{n \in \nat}$. 

\begin{defn}\label{def: unif dense}(Uniformly dense sampling)
We say that the sequence of sets of sample inputs $(G_n)_{n\in\nat}$ becomes uniformly dense relative to $\inspace$ at a rate $r$ (denoted by $(G_n) \stackrel{r}{\bdu} \inspace$) if $\exists r:\nat \to \Real_+$ such that $\lim_{n\to \infty} r(n) = 0$ and $\forall n\in\nat$, $\sup_{x \in \inspace} \inf_{s_n \in G_n} \metric(s_n,x) \leq r(n)$.
\end{defn}

Using this definition, we can provide the following asymptotic guarantee for the general Lipschitz interpolation method.

\begin{thm} \label{lemma:general}
Suppose Assumptions \ref{ass:noise lip simple} and \ref{ass:outspace norm} hold, $\inspace$ is bounded
and the target function $f\in Lip(L^*,\metric)$\footnote{unless specified otherwise, Lipschitz continuity will be assumed to be w.r.t. the metrics $\metric$,$\metric_\outspace$ on the spaces $\inspace, \outspace$.} with best Lipschitz constant $L^* \in \Real_+$. If the sampling set sequence $(\data_n)_{n \in \nat}$ has sample inputs  $(G_n^\inspace)_{n \in \nat}$ such that $\exists r \in o(1) : (G_n^\inspace) \stackrel{r}{\bdu} \inspace$ and the sequence of predictors $(\hat f_n)_{n \in \nat}$ are computed by a general Lipschitz interpolation framework with a hyperparameter $L \in \Real_{\geq 0}$ set such that $L \geq L^*$ then we have: 
$$\forall \epsilon >0 \text{, } \lim_{n \to \infty}\pmeas \left( \sup_{x\in \inspace}\metric_\outspace(\hat f_n(x), f(x)) > \epsilon \right) = 0 .$$
\end{thm}
Before providing the proof of Theorem \ref{lemma:general}, we recall the notion of $\epsilon$-covering that will be used in multiple proofs of this paper.
\begin{defn}\label{def: cover}($\epsilon$-Cover) Let $d \in \nat$, $\epsilon>0$ and consider a set $\inspace \subset \Real^d$ and a metric $\metric$ on $\Real^d$. Denoting $B_\epsilon(x)$ the ball of radius $\epsilon$ centred in $x \in \inspace$ with respect to $\metric$, we define an $\epsilon$-cover of $\inspace$ as a subset $Cov(\epsilon)\subset \Real^d$ such that 
$\inspace \subset \bigcup_{x \in Cov(\epsilon)} B_\epsilon(x)$ and the associated set of balls as $\mathcal{B}:= \{B_\epsilon(x) | x \in Cov(\epsilon) \}$. We say furthermore that $ Cov(\epsilon)$ is a $\epsilon$-minimal cover of $\inspace$ if $|Cov(\epsilon)| = min\{n:\exists \epsilon\text{-covering over \inspace of size n}\}.$ 
\end{defn}

\begin{proof}
We begin by establishing a general bound on $\metric_\outspace(\hat f_n(x), f(x))$, $\forall n \in \nat$, $\forall x \in \inspace$. For any $x \in \inspace$ we have:
\begin{align*}
    & \metric_\outspace(\hat f_n(x), f(x)) =|\hat f_n(x) - f(x)| \\
    & =  \left| \frac{1}{2} \min_{i=1,...,N_n} \{ \tilde f_i + L\metric(x,s_i)\}  +  \frac{1}{2}\max_{i=1,...,N_n} \{ \tilde f_i  - L\metric(x,s_i)\} - f(x) \right| \\
    & =  \left| \frac{1}{2} \min_{i=1,...,N_n} \{ \tilde f_i - f(x) + L\metric(x,s_i)\}  +  \frac{1}{2}\max_{i=1,...,N_n} \{ \tilde f_i - f(x) - L\metric(x,s_i)\} \right|.
\end{align*}
Using the Lipschitz continuity of $f$
, we obtain the following set of inequality relations for the two terms stated above:
\begin{align*}
& (1)     \min_{i=1,...,N_n}\{e_i\} \leq \min_{i=1,...,N_n} \left\{ \tilde f_i - f(x) + L\metric(x,s_i) \right\} \nonumber \leq  \min_{i=1,...,N_n}\{e_i+ (L^*+L)\metric(x,s_i) \}.
\\
& (2)
    \max_{i=1,...,N_n} \left\{e_i-(L^*+L)\metric(x,s_i) \right\}
    \leq  \max_{i=1,...,N_n} \left\{ \tilde f_i-f(x) - L\metric(x,s_i) \right\} \leq  \max_{i=1,...,N_n}\{e_i\}. \nonumber
\end{align*}
In combination, we see that
\begin{align*}
& |\hat f_n(x)- f(x)| \\
& \leq \frac{1}{2} \max \Big\{  \underbrace{\max_{i=1,...,N_n}\{e_i\}+
\min_{i=1,...,N_n} \left\{e_i 
+(L^*+L) \metric(x,s_i) \right\} }_{(I)}  ,
\\ 
& \underbrace{-\min_{i=1,...,N_n}\{e_i\} - \max\limits_{i=1,...,N_n} \left\{ e_i  -(L^*+L)\metric(x,s_i) \right\}}_{(II)} \Big\} .
\end{align*}
$(I), (II)$ can then be bounded using the assumption of uniform convergence of the grid (see Definition \ref{def: unif dense}). 
%
Define $R:=\frac{\epsilon}{4(L^*+L)}$ and consider the minimal covering of $\inspace$ of radius $R$ with respect to $\metric$ that we denote $Cov(R)$ and the associated set of hyperballs $\mathcal{B}$. By uniform convergence of the sample inputs, there exists $M\in\nat$ such that $\forall n>M$: $\forall B \in \mathcal{B}$, $|B\cap G_n^\inspace| > 0$. Then, the following upper bound holds for $(I)$ with $n>M$ ($(II)$ can be bounded in a similar way):
\begin{align*}
    & \max_{i=1,...,N_n}\{e_i\}+
\min_{i=1,...,N_n}\{e_i
+(L^*+L) \metric(x,s_i) \}   \nonumber\\
    \leq & \max_{i=1,...,N_n}\{e_i\}+
\min_{s_i \in B^x \cap G_n^\inspace }\{e(s_i)
+(L^*+L) \metric(x,s_i)  \} \\
    \leq &  \max_{i=1,...,N_n}\{e_i\} + \min_{s_i \in B^x \cap G_n^\inspace} \left \{e(s_i) + 2(L^*+L)R \right\}
\end{align*}
where with abuse of notation, $e(s_i)$ denotes the noise term associated with the input $s_i$ and $B^x$ denotes a hyperball $B \in \mathcal{B}$ such that $x\in B$. Similarly for $(II)$, we obtain 
$$
(II) \leq    \max_{i=1,...,N_n}\{-e_i\} + \min_{s_i \in B^x \cap G_n^\inspace}\{-e(s_i) + 2(L^*+L)R\}.
$$
Let $\epsilon>0$.
Utilising these bounds, $\forall n>M$, we obtain
\begin{align*}
   & \pmeas (\sup_{x \in \inspace}\metric_\outspace(\hat f_n(x), f(x)) > \epsilon)   
   \\
   &
   \leq \pmeas \Big( 2(L^*+L)R 
   +
   \frac{1}{2} \sup_{x \in \inspace}\max \Big\{ \max_{i=1,...,N_n}\{e_i\} 
 + 
\min_{s_i \in B^x \cap G_n^\inspace }\{e(s_i)\}, \\
&
 - \min_{i=1,...,N_n}\{e_i\} - \max_{s_i \in B^x \cap G_n^\inspace }\{e(s_i)\} \Big\} > \epsilon \Big)
\\
 & \leq \pmeas \Big( 
   \frac{1}{2} \max_{B \in \mathcal{B}}\max\Big\{ \max_{i=1,...,N_n}\{e_i\} 
 + 
\min_{s_i \in B \cap G_n^\inspace }\{e(s_i)\}, \\
&
 -\min_{i=1,...,N_n}\{e_i\} - \max_{s_i \in B \cap G_n^\inspace }\{e(s_i)\} \Big\} > \epsilon - 2(L^*+L)R \Big)
\\
& 
\leq  \pmeas \left(\max_{B \in \mathcal{B}} \max_{i=1,...,N_n}\{e_i\} 
 + \min_{s_i \in B \cap G_n^\inspace}\{e(s_i)\}> \frac{\epsilon}{2} \right) 
 \\
& + \pmeas\left( \max_{B \in \mathcal{B}} \max_{i=1,...,N_n}\{- e_i\} + \min_{s_i \in B \cap G_n^\inspace}\{- e(s_i)\} \} > \frac{\epsilon}{2} \right) 
\end{align*}
where the last inequality follows by definition of $R$.
Both probability terms stated above can be shown to converge to $0$ as follows:
$$
\pmeas \left(\max_{B \in \mathcal{B}} \max_{i=1,...,N_n}\{e_i\}  + \min_{s_i \in B \cap G_n^\inspace}\{e(s_i)\}> \frac{\epsilon}{2} \right)
 $$
 $$
 = 1 - \pmeas \left(\max_{B \in \mathcal{B}} \max_{i=1,...,N_n}\{e_i\} 
 + \min_{s_i \in B \cap G_n^\inspace}\{e(s_i)\} \leq  \frac{\epsilon}{2} \right)
$$
 $$
 = 1 - \pmeas \left( \forall B \in \mathcal{B}, \max_{i=1,...,N_n}\{e_i\} 
 + \min_{s_i \in B \cap G_n^\inspace}\{e(s_i)\} \leq \frac{\epsilon}{2} \right)
$$
 $$
 \leq 1 - \pmeas \left( \forall B \in \mathcal{B}: \max_{i=1,...,N_n}\{e_i\} \in I_1, \min_{s_i \in B \cap G_n^\inspace}\{e(s_i)\} \in I_2 \right).
$$
Here $I_1 := [\obserrbnd - \frac{\epsilon}{4}, \obserrbnd]$ and $I_2 := [- \obserrbnd , -\obserrbnd + \frac{\epsilon}{4} ]$. Applying a similar argument to the one given in $(\star\star)$ in the proof of Theorem \ref{thm:conv rate}, we have that the last term is upper bounded by 
$$
  1 - \prod_{B \in \mathcal{B}} \pmeas \left( \max_{i=1,...,N_n}\{e_i\} \in I_1, \min_{s_i \in B \cap G_n^\inspace}\{e(s_i)\} \in I_2 \right)
$$
$$
\leq 1 - \pmeas \left( \max_{i=1,...,N_n}\{e_i\} \in I_1, \min_{i=1,...,L_n}\{e_i\} \in I_2 \right)^{|\mathcal{B}|}
$$
where $L_n:= \min_{B\in\mathcal{B}} |B \cap G_n^\inspace|$ and $|.|$ is used to denote the cardinality operator for finite sets. By the uniformity of the convergence of the sample inputs, we have that $\lim_{n\to\infty}L_n = \lim_{n\to\infty}N_n =+\infty$. Using basic identities of probability theory and  applying Assumption \ref{ass:noise lip simple}, we have that 
$$\lim_{n\to\infty} \pmeas \left( \max_{i=1,...,N_n}\{e_i\} \in I_1, \min_{i=1,...,L_n}\{e_i\} \in I_2 \right) = 1$$
which implies that 
$$\lim_{n \to \infty}\pmeas \left( \sup_{x\in \inspace}\metric_\outspace(\hat f_n(x), f(x) ) > \epsilon \right) = 0$$
and concludes the proof.
\end{proof}

Theorem \ref{lemma:general} ensures that the classical Lipschitz interpolation method is asymptotically consistent for a general selection of input metrics. Furthermore, a similar result for Lipschitz interpolation with a multi-dimensional output setting $\outspace \subset R^m$ for $m\in\nat$ follows naturally by applying Theorem \ref{lemma:general} to each output component function (the noise assumption would need to be modified in this case; e.g. see Assumption \ref{ass:noise lip2}).

In general, we are mostly interested in simple metric choices for $\metric$. In this case with additional assumptions on $\inspace$ and $\outspace$, we can extend the result obtained in Theorem \ref{lemma:general} by deriving asymptotic rates of convergence for the general Lipschitz interpolation method. More precisely, we have the following definition (\cite{gyorfi2002distribution}):

\begin{defn} Consider a sequence of non-parametric predictors $( \hat f_n)_{n \in \nat}$ and a class of functions $\mathcal{C}$ endowed with a norm $\norm{.}$. Let $(a_n)_{n \in \nat}$ be a sequence of positive constants in $\Real$. We define $(a_n)_{n \in \nat}$ as the rate of convergence of $( \hat f_n)_{n \in \nat}$ on $\mathcal{C}$ with respect to $\norm{.}$ if there exists $c > 0$ such that
$$ 
\limsup_{n \to \infty}\sup_{f \in \mathcal{C}} \mathbb{E} \left[a_n^{-1}\norm{\hat f_n - f} \right] = c < \infty.
$$
\end{defn}

In order to avoid extreme cases of compact spaces, the following general assumption provides a light geometric assumption on $\inspace$.
\begin{ass}\label{ass: input space2}(Geometric Assumption on $\inspace$) Let $\inspace\subset \Real^d$ be compact and convex. There exist two constants $r_0>0, \theta \in(0,1]$ such that  $\forall x \in \mathcal{X}$, $r \in\left(0, r_0\right): \newline \operatorname{vol}\left(B_r(x) \cap \mathcal{X}\right) \geq \theta \operatorname{vol}\left(B_r(x)\right) $.
\end{ass}
 Assumption \ref{ass: input space2} has been used in the learning theory literature (e.g. see \cite{hu2020smooth} \cite{bachoc2021instance}) and ensures that for all $x\in\inspace$, a constant fraction of  ball with a sufficiently small radius and centred in $x$ is contained in $\inspace$. For example, if $\inspace$ is a the unit hypercube then Assumption \ref{ass: input space2} holds with $r_0 = 1, \theta = 2^{-d}$. 

The additional assumptions on the sampling of the sample inputs $(D_{n})_{n\in\nat}$ and metric of the input space $\metric$ are relatively standard and are given as follows:

\begin{ass}\label{ass:sampling}(Assumption on Sampling) $(G_n^\inspace)_{n \in \nat}$ is a randomly sampled sequence on $\inspace$ with a sampling distribution density that is bounded away from zero on $\inspace$. 
\end{ass}
\begin{ass}\label{ass:inspace norm}(H\"older Condition) We restrict the input space metrics under consideration to be of the form $\metric(x,y) = \norm{x-y}_p^\alpha$
where $\alpha \in (0,1]$ and $\norm{.}_p$ denotes the usual $p$-norm on $\Real^d$ with $p\in\nat \cup \{+\infty\}$.
\end{ass}

\begin{thm}\label{thm:conv rate}
Consider an input space $\inspace\subset \Real^d$ that satisfies Assumption \ref{ass: input space2}, an output space $\outspace \subset \Real$ and the function space $\mathcal{C} = Lip(L^*,\metric)$ with $L^* \in \Real_{\geq 0}$ endowed with the supremum-norm: $\norm{h}_\infty = \sup_{x \in \inspace}\norm{h(x)}_\mathcal{Y}$. Assume that Assumptions \ref{ass:noise lip simple}, \ref{ass:noise lip}, \ref{ass:sampling}, \ref{ass:inspace norm} with $\alpha \in (0,1]$, $p\in \nat$ hold.
Then, any sequence of predictors $(\hat f_n)_{n \in \nat}$  generated by the general Lipschitz interpolation framework set with a hyperparameter $L \geq L^*$ achieves a rate of convergence of at least $(a_n)_{n\in\nat} := \left((n^{-1}log(n) )^\frac{\alpha}{d+\eta\alpha} \right)_{n \in \nat}$ with respect to $\norm{.}_\infty$, i.e.
$$ 
\limsup_{n \to \infty}\sup_{f \in Lip(L^*,\metric)} \mathbb{E}\left[a_n^{-1}\norm{\hat f_n - f}_\infty \right] < \infty.
$$
\end{thm}
\begin{proof}
    See appendix \ref{sec: proof of thm: conv rate}.
\end{proof}

Convergence lower bounds do not exist for the exact setting considered in this paper signifying that we cannot directly compare the rates stated in Theorem \ref{thm:conv rate} to a theoretically optimal convergence rate. Instead, we can note that the convergence rate of Lipschitz interpolation is in line with several known optimal rates in related settings (see Table \ref{table:conv rate}), i.e. non-parametric regression on the Lipschitz continuous function space endowed with an $L_2$ or $L_\infty$ norm. In particular, we note that the exponent of the convergence rate derived for Lipschitz interpolation exactly matches the exponent of the convergence rate derived in \cite{tsybakov2004introduction} in the case where the noise distribution is assumed to be uniform (i.e. $\eta = 1$). Our convergence rate is however larger by a log-factor due to a difference in norm.

Furthermore, by varying $\eta$ in Assumption \ref{ass:noise lip}, we can compare our rate of convergence: O$\left((n^{-1}log(n) \right)^\frac{\alpha}{d+\eta\alpha} )_{n \in \nat}$ to classical non-parametric convergence rates. More precisely, we observe the following:

\footnotetext[4]{\label{footnote: gauss}Various generalisations of this noise assumption exist, see \cite{stone1982optimal}.}

\begin{itemize}
\item {\underline{For $\eta<2$}}: the derived convergence rates for Lipschitz interpolation are better than the known optimal convergence rates obtained under a Gaussian tail noise assumption on the noise distribution: $(n^{-1}log(n))^{\frac{\alpha}{2\alpha + d}}$ (\cite{stone1982optimal}) which are attained\footnote{\label{footnote: gyorfi conv}Note that these methods can be shown to converge at this rate under the simple assumption of bounded variance (\cite{gyorfi2002distribution}).} by Gaussian process regression (\cite{yang2017frequentist}) and other kernel-based non-parametric methods such as local polynomial regression (\cite{stone1982optimal}) or the Nadaraja-Watson estimator (\cite{tsybakov2004introduction}, \cite{muller2010estimation}). 

\item \underline{For $\eta>2$}: the opposite becomes true and these alternative non-parametric methods can be expected to converge quicker asymptotically than the Lipschitz interpolation framework. 
\end{itemize}

\begin{table*}
\centering
\begin{tabular}{||c c c c||} 
 \hline
 Algorithm/Type & Convergence Rate & Noise Assumption & Norm \\ [0.5ex] 
 \hline\hline
 LI (Upper Bound) &  O$\left(n^{-1}log(n) \right)^\frac{\alpha}{d+\eta\alpha}$& Bounded & $L_\infty$. \\ 
  \hline
 Optimal (\cite{tsybakov2004introduction}) &  $\Theta \left(n^{-1} \right)^\frac{\alpha}{d+\alpha}$ & Uniform ($\eta=1$) & $L_2$ \\ 
 \hline
 Optimal (\cite{stone1982optimal}) &  $ \Theta \left(n^{-1}log(n) \right)^\frac{\alpha}{d+2\alpha}$ & Gaussian \footnotemark[4] & $L_\infty$ \\ 
 \hline
 Optimal (\cite{stone1982optimal}) &  $\Theta \left(n^{-1} \right)^\frac{\alpha}{d+2\alpha}$ & Gaussian\color{blue}\footnotemark[4]{\label{footnote: gauss}}\color{black} & $L_2$  \\
 \hline
Optimal (\cite{jirak2014adaptive}) &  $\Theta \left(n^{-1} \right)^\frac{\alpha}{1+\eta\alpha}$ (d=1) & Boundary Regr. & $L_2$  \\
 \hline
 Upper Bound (\cite{selk2022multivariate}) &  O$ \left(n^{-1}log(n) \right)^\frac{\alpha}{d+\eta\alpha}$ & Boundary Regr. & $L_\infty$  \\
 \hline
\end{tabular}
\caption{Comparison of the convergence rate derived in Theorem \ref{thm:conv rate} with optimal rates of convergence rates in similar settings and discussion given in this section.}
\label{table:conv rate}
\end{table*}

This "$\eta$-condition" provides a theoretical tool for comparing the expected long-run performance of Lipschitz interpolation relative to alternative non-parametric methods and can help guide the choice of the system identification approach if information on the non-regularity of the noise distribution is obtainable. We note that the convergence rates of the kernel-based non-parametric methods stated in Table \ref{table:conv rate} hold under general noise assumptions (see footnotes \ref{footnote: gauss} and \ref{footnote: gyorfi conv} below) and that, aside from the Nadaraja-Watson estimator, no formal derivation of improved convergence rates in the bounded noise setting considered in this paper currently exists\footnote{To the extent of our knowledge.} for these methods. As these kernel-based non-parametric frameworks generally rely on local averaging of the noise in order to prove convergence, it is expected that their convergence rates do not improve with respect to their classical convergence rates (stated in Table \ref{table:conv rate}) under Assumption \ref{ass:noise lip simple} and  Assumption \ref{ass:noise lip}. This has been formally shown to be true for the Nadaraja-Watson estimator by \cite{muller2010estimation} and a more general discussion on the topic can be found in \cite{meister2013asymptotic}.

\begin{figure*}[t]
\centering
\includegraphics[width=0.32\textwidth]{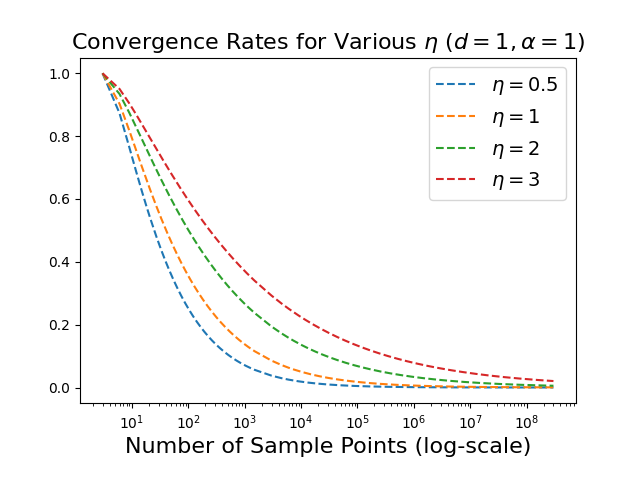}  
\includegraphics[width=0.32\textwidth]{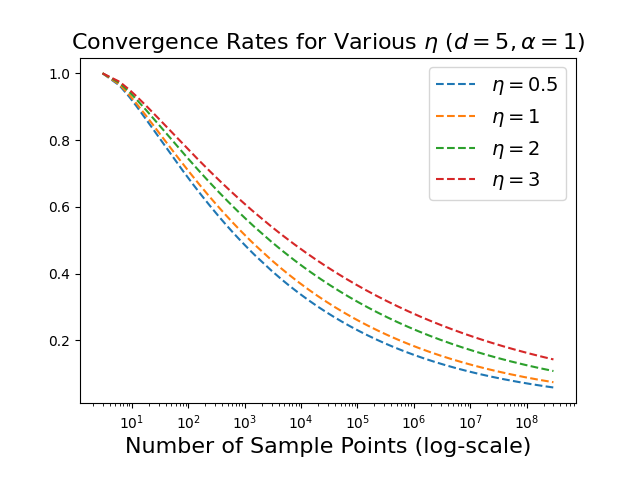}
\includegraphics[width=0.32\textwidth]{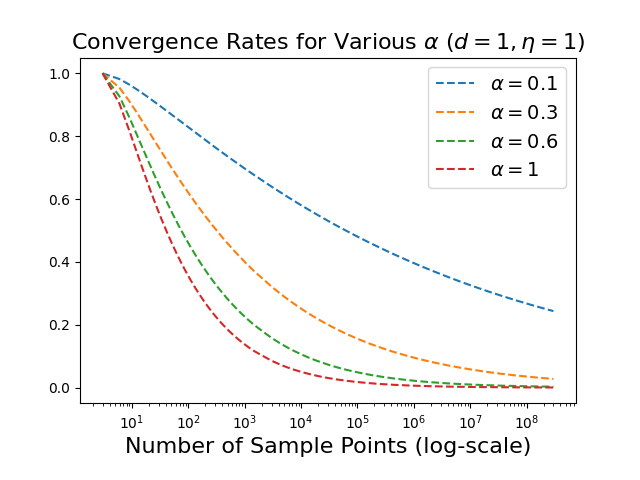}  
	\caption[Illustration of the behaviour of the convergence rates derived in Theorem \ref{thm:conv rate} for various values of $(d, \alpha, \eta)$]{ Illustration of the behaviour of the convergence rates derived in Theorem \ref{thm:conv rate} for various values of $(d, \alpha, \eta)$.	\label{Fig: conv rate}}%
\end{figure*}

As discussed above, the convergence rates obtained in Theorem \ref{thm:conv rate} under the bounded noise assumptions are better than the classical optimal convergence rates derived by \cite{stone1982optimal}. This is possible as the lower bounds of these optimal convergence rates are generally derived under the condition that the noise has a positive density with respect to the Lebesgue measure on $\Real$ which does not hold for the noise assumptions of this paper. As a consequence, $O \left(n^{-1}log(n) \right)^\frac{\alpha}{d+\eta\alpha}$ provides a new general upper bound on the non-parametric regression problem in the bounded noise setting and future work can be done on deriving lower bounds to match these results. We expect the lower bounds to be tight given recent results by \cite{jirak2014adaptive} on the optimal convergence rates of the related non-parametric boundary regression problem (see below for a more detailed discussion).


The optimality results of Theorem 2 of \cite{milanese2004set} show that $( \tilde \decke_n)_{n \in \nat}$,
$( \tilde \boden_n)_{n \in \nat}$  are exactly equal to the upper and lower bounds of the \textit{feasible systems set}, i.e.
the set of all data-consistent Lipschitz continuous systems and therefore provide worst-case error prediction bounds. With little modification to the proof of Theorem \ref{thm:conv rate}, both error bounds can be shown to converge to $f$ at the same rate as $(\hat f_n)_{n \in \nat}$ as stated in the following Corollary.

\begin{cor} \label{cor:FSS}
    Assume that the setting and assumptions of Theorem \ref{thm:conv rate} holds. The worst-case prediction guarantees $( \tilde \decke_n)_{n \in \nat}$,
$( \tilde \boden_n)_{n \in \nat}$ defined in Remark \ref{rem:alt} with second hyperparameter: $\obserrbnd' = \obserrbnd$, converge uniformly to any target function $f\in Lip(L^*,\metric)$ at a rate of at least $((n^{-1}log(n) )^\frac{\alpha}{d+\eta\alpha} )_{n \in \nat}$.
\end{cor}
\begin{proof}
    Follows from the proof of Theorem \ref{thm:conv rate}.
\end{proof}

A connection between our convergence results and succeeding work on non-parametric boundary regression (see \cite{hall2009nonparametric} and ensuing works) can be made. More precisely, consider the predictive functions $(\tilde \decke_n')_{n \in \nat}$, $(\tilde \boden_n')_{n \in \nat}$ defined for all $n\in \nat$ as $\tilde \decke_n': \inspace \to \outspace$, $x \mapsto  \decke_n(x) + \obserrbnd_1+\obserrbnd_2$ and  $\tilde \boden_n': \inspace \to \outspace$, $x \mapsto  \boden_n(x) - \obserrbnd_1-\obserrbnd_2$
 where $\obserrbnd_1, \obserrbnd_2$ are tight asymmetric bounds on the noise. 
 $(\tilde \decke_n')_{n \in \nat}$, $(\tilde \boden_n')_{n \in \nat}$ can be interpreted as conservative non-parametric boundary regression methods. Therefore, in the context of bounded noise, the two problems are equivalent and we can again slightly modify the proof of Theorem \ref{thm:conv rate} to obtain the same uniform asymptotic convergence rates of $O(n^{-1}log(n) )^\frac{\alpha}{\eta d+\alpha} )$ as $(\hat f_n)_{n \in \nat}$. These rates exactly match the recently derived best convergence rates in the multivariate boundary regression problem (\cite{selk2022multivariate}) and have the same exponent\footnote{They differ by a log-factor which is usual when considering the $L_\infty$ norm instead of the $L_2$ norm.} as the optimal rates derived with respect to the $L_2$ norm (\cite{jirak2014adaptive}). In order to properly define $(\tilde \decke_n')_{n \in \nat}$,
$(\tilde \boden_n')_{n \in \nat}$, prior knowledge of an upper bound on the Lipschitz constant ($L\geq L^*$) as well as the H\"older exponent ($\alpha$) and of tight bounds of the noise ($\obserrbnd_1,\obserrbnd_2$) are needed. However in contrast to the proposed "best" non-parametric estimators that attain the optimal rates, we do not require prior knowledge of the degree of "non-regularity" of the noise ($\eta$, defined in Assumption \ref{ass:noise lip}) which is usually required in order to define an optimal bandwidth hyperparameter (\cite{10.3150/17-BEJ992}, \cite{selk2022multivariate}). In the bounded noise setting, our assumption is therefore arguably more natural and simpler to verify in practice as ($\eta$) is generally hard to determine precisely.

\section{Online Learning: Asymptotics} \label{sec:online}

A set-up not yet explicitly considered in this paper but relevant to control applications is when the output variables can be used as input variables. More specifically, we consider the case where $f$ models the dynamics of a non-linear autoregressive stochastic system with exogenous control variables:
$$y_n = f(x_n) + e_n $$
where $x_n = (y_{n-d_y},...,y_{n-1},u_{n-d_u}, ... ,u_{n})$ with $y_i\in\outspace\subset\Real$ and $u_i \in \mathcal{U} \subset \Real^s$
for $d_y, d_u, s \in \nat$
and $e_n \in \Real$ is a noise variable that satisfies Assumption \ref{ass:noise lip simple}.  Here, $y_i$ denotes the autoregressive inputs and the $u_i$ denote vectors of past and current control inputs.
In this setting, we will therefore consider $\inspace = \Real^{d_y}\times \mathcal{U}^{d_u+1} \subset \Real^{d_y + (d_u+1)(s)}, \outspace = \Real$.  If the dynamics and control inputs are such that the underlying dynamical system is ergodic then Theorem \ref{lemma:general} can be applied and a weaker version of Theorem \ref{thm:conv rate} can be derived. However, in general, this cannot be guaranteed and the following result on the asymptotic point-wise convergence of the general Lipschitz interpolation framework is needed. 

\begin{cor} \label{cor:online} Consider $\inspace, \outspace, (x_n)_{n \in \nat}, (u_n)_{n \in \nat}, (y_n)_{n \in \nat}$ as defined above, $L^*\geq 0$ and $ (\hat f_n)_{n \in \nat}$ as defined in Definition \ref{def:KILsimplified} with $L\geq L^*$ and $(\mathcal{D}_n)_{n\in\nat} = (x_n, y_n)_{n\in\nat}$. Suppose that Assumptions \ref{ass:noise lip simple}, \ref{ass:outspace norm} hold. Assume furthermore that $\mathcal{U}\subset \Real^{s}$ is bounded. Then $\forall p \in \nat$, $M^*\in\Real^+$
$$\lim_{n \to \infty} \sup_{f\in \overline{Lip}(L^*, \metric, M^*)} \mathbb{E} \left[\norm{f(x_{n+1}) - \hat f_n(x_{n+1})}_\outspace^p \right] \to 0 $$
where $\overline{Lip}(L^*, \metric, M^*)$ denotes the set containing all functions in $Lip(L^*, \metric)$ that are bounded by $M^*$, i.e. $\norm{f}_\infty \leq M^*$.
\end{cor}
\begin{proof} As in the proof of Theorem \ref{thm:conv rate},
we have that for all $n \geq 1$ and any sampling procedure $\mathcal{D}_n$, $\sup_{f\in \overline{Lip}(L^*, \metric, M^*)} \norm{\hat f_n - f}_\infty$ is uniformly bounded with probability $1$. This follows from (1) the existence of a bounded set $\tilde \inspace \subset \inspace$ such that $(x_n)_{n\in\nat} \subset \tilde \inspace$ (with probability 1) which is due to the boundedness of $\overline{Lip}(L^*, \metric, M^*)$, the compactness of $\mathcal{U}$ and Assumption \ref{ass:noise lip simple} (which implies that the noise is bounded), (2) $f\in Lip(L^*, \metric)$ and (3) by construction of the Lipschitz interpolation framework. More precisely, we have $\forall n \in \nat$, $\sup_{f\in \overline{Lip}(L^*, \metric, M^*)}\norm{ \hat f_n - f}_\infty\leq 2\obserrbnd + 2L\delta_{\metric} (\tilde \inspace)$  where $\delta_{\metric} (\tilde \inspace):= \sup_{x,y\in\tilde\inspace}\metric(x,y)$. Using Lemma \ref{lemma:tech conv equiv}, it is therefore sufficient to show convergence in probability, i.e.
$$
\forall \epsilon>0: \text{ } \lim_{n \to \infty}\sup_{f \in \overline{Lip}(L^*,\metric, M^*)}\pmeas(|\hat f_n(x_{n+1}) - f(x_{n+1})| > \epsilon ) = 0
$$
which can be done through a modified proof of Theorem \ref{lemma:general} as follows. Fix $\epsilon > 0$ and consider the minimal covering of $\bar \inspace$ by balls of radius $r < \frac{\epsilon}{4(L^*+L)}$ which we denote $cov(r)$ and the associated set of hyperballs $\mathcal{B}$ (the existence of a finite covering is guaranteed by the boundedness of $\tilde \inspace$). There exists $N_1 \in \nat$ such that for all $B\in\mathcal{B}$; $(x_n)_{n \geq N_1} \cap B \in \{0, +\infty \}$. Denoting by $ \mathcal{ \tilde B} \subset \mathcal{B}$ the subset of $\mathcal{B}$ consisting of hyperballs that contain an infinite number of elements of $(x_n)_{n \geq N_1}$, we can proceed as in the proof of Theorem \ref{lemma:general}. 

Let $f\in \overline{Lip}(L^*, \metric, M^*)$ be arbitrary. For $n>N_1$ sufficiently large (such that there is at least one sample input in each hyperball of $\mathcal{ \tilde B}$), applying the same arguments as in the proof of Theorem \ref{lemma:general}:
$$
\pmeas \left(\metric_\outspace(\hat f_n(x_{n+1}),  f(x_{n+1})) > \epsilon \right)  
$$
$$
\leq \pmeas \left(\max_{B\in\mathcal{ \tilde B}} \left|\max_{i=1,...,n}\{e_i\} 
 + \min_{x_i\in B}\{e(x_i)\} \right|> \frac{\epsilon}{2} \right) 
+ \pmeas \left(\max_{B\in\mathcal{ \tilde B}} \left|\min_{i=1,...,n}\{e_i\} + \max_{x_i\in B} \{e(x_i)\} \right|  > \frac{\epsilon}{2} \right).
$$

As the choice of $f \in \overline{Lip}(L^*, \metric, M^*)$ was arbitrary and the upper bound expressed above does not depend on $f$, we have that both terms of this upper bound can be treated with the same approach as the one used to conclude the proof of Theorem \ref{lemma:general}. This implies 
\begin{align*}
 & \quad \lim_{n \to \infty}\sup_{f \in \overline{Lip}(L^*,\metric, M^*)}\pmeas \left(|\hat f_n(x_{n+1}) - f(x_{n+1})| > \epsilon \right) \\
& \leq \lim_{n \to \infty} \pmeas \left(\max_{B\in\mathcal{ \tilde B}}|\max_{i=1,...,n}\{e_i\} 
 + \min_{x_i\in B}\{e(x_i)\}|> \frac{\epsilon}{2} \right)  \\
 & + \lim_{n \to \infty}\pmeas \left(\max_{B\in\mathcal{ \tilde B}}|\min_{i=1,...,n}\{e_i\} + \max_{x_i\in B} \{e(x_i)\}|  > \frac{\epsilon}{2} \right) \\
 & = 0
\end{align*}
which concludes the proof.
\end{proof}

The setting considered in Corollary \ref{cor:online} is the same as the one considered in \cite{milanese2004set} and in ensuing applications of the Lipschitz interpolation framework in the context of MPC (see \cite{canale2014nonlinear}, \cite{manzano2020robust}). As in Corollary \ref{cor:FSS}, the worst-case prediction guarantees $( \tilde \decke_n)_{n \in \nat}$,
$( \tilde \boden_n)_{n \in \nat}$ can be shown to provide similar guarantees to the one proposed Corollary \ref{cor:online} which provides a theoretical guarantee that even conservative adaptive controllers relying on worst-case bounds of Lipschitz interpolation methods will consider the true underlying dynamics in the long run. 
In Section \ref{sec: online control}, a slight modification of Corollary \ref{lem:mult online} that considers dynamics with multidimensional outputs $(y_n)_{n\in\nat}$ is given. This extension is then applied in the context of tracking control in order to obtain closed-loop stability guarantees for a simple
online-learning based controller


To conclude this section, we remark that if an additional assumption is made on the sequence of inputs $(x_n)_{n\in\nat}$, then the convergence rate derived in Theorem \ref{thm:conv rate} holds in the online learning setting.
This assumption is given using the following definition on the "regularity of the sampling" of $(x_n)_{n\in\nat}$.
\begin{defn}\label{def:reg samp}(Regularity Assumption for $(x_n)_{n\in\nat}$)
We say that $(x_n)_{n\in\nat}$ is regularly sampled on a set $\bar \inspace \subset \inspace $ if $\exists N\in\nat$,
$(x_n)_{n\in\nat_{\geq N}} \subset \bar \inspace$ and
$\exists M\in\nat$ such that $\forall n>N$ and $\forall A\subset \bar \inspace$ ,
$$\pmeas(x_{n+M}\in A |x_n)> C\mu(A)$$   
where $\mu(A)$ denotes the Lebesgue measure of $A$ and $C>0$ is an arbitrary constant.
\end{defn}
In essence, Definition \ref{def:KILsimplified} states that $(x_n)_{n\in\nat}$ is regularly sampled on a given set $\bar \inspace \subset \inspace$ if $(x_n)_{n\in\nat}$ will eventually be contained in $\bar \inspace$ and will continue to visit all of $\bar \inspace$ with non-zero probability. The existence of such a set depends implicitly on the target function and the defined control inputs.

\begin{cor}\label{cor:online conv rate}
    Assume that the setting and assumptions of Corollary \ref{cor:online} hold and consider $f\in \overline{Lip}(L^*,\metric, M^*)$. If Assumption \ref{ass:noise lip} holds and the stochastic control law $u_{n+1} := u(x_n,\hat f_n, \mathcal{D}_n)$ is defined such that $(x_n)_{n\in\nat}$ is  regularly sampled on a bounded set $\bar \inspace \subset \inspace$ that satisfies Assumption \ref{ass: input space2}, then 
$$ 
\limsup_{n \to \infty} \mathbb{E}[a_n^{-1}\norm{f(x_{n+1}) - \hat f_n(x_{n+1})}_\outspace ] < \infty.
$$
where $(a_n)_{n\in\nat} := ((n^{-1}log(n) )^\frac{\alpha}{d+\eta\alpha} )_{n \in \nat}$.
\end{cor}
\begin{rem}
    From the proof of Corollary \ref{cor:online}, we have that 
    if $\mathcal{U}$ is bounded and $f\in \overline{Lip}(L^*, \metric, M^*)$, then there exists a bounded $\tilde \inspace \subset \inspace$ that contains $(x_n)_{n\in\nat}$ with probability $1$. Therefore, only the second part of Definition \ref{def:reg samp} and the geometric shape of $\bar \inspace$ need to be checked in order for Corollary \ref{cor:online conv rate} to hold.
\end{rem}
\begin{proof}
    The proof of Corollary \ref{cor:online conv rate} follows from Theorem \ref{thm:conv rate}. More precisely:
    
    By assumption, we have that there exists $M, N\in\nat$ and a bounded set $\bar \inspace \subset \inspace$ such that Definition \ref{def:reg samp} and Assumption \ref{ass: input space2} hold. Consider the sequence $(x_n)_{n\in\nat{\geq N}} \subset \bar \inspace$ and the subsequence $(\tilde x_n)_{n\in\nat}  \subset (x_n)_{n\in\nat_{\geq N}}$  defined such that $\tilde x_n = x_{Mn+N}$ for all $n\in \nat$. From Definition \ref{def:reg samp}, we have that for all $n\in\nat$, $\tilde x_n$ is sampled on $\bar \inspace$ with a probability distribution whose density is bounded away from zero on all of $\bar \inspace$.

    Then, defining $(\hat f_n^M)_{n\in\nat}$ as the predictors of the Lipschitz interpolation framework with hyperparameter $L$ and sample inputs $(\tilde x_n)_{n\in\nat}$, we can apply Theorem \ref{thm:conv rate} to $(\hat f_n^M)_{n\in\nat}$. This implies that $(\hat f_n^M)_{n\in\nat}$ converges uniformly on $\bar \inspace$ to $f$ at a rate that is upper bounded by $(a_{\floor{\frac{n}{M}}})_{n\in\nat} = \tilde c (a_{n})_{n\in\nat}$ for some $\tilde c>$ that depends on $M$ and where $n\in\nat$ denotes the index of the original sequence: $(x_n)_{n\in\nat}$. As the asymptotic convergence rate of $(\hat f_n)_{n\in\nat}$ is at least as fast as the convergence rate of $(\hat f_n^M)_{n\in\nat}$ due to the fact that the input samples utilised by $(\hat f_n^M)_{n\in\nat}$ are also utilised by $(\hat f_n)_{n\in\nat}$, we have that $(\hat f_n)_{n\in\nat}$ achieves the same uniform convergence rate on $\bar \inspace$. Finally, as $(x_n)_{n\in\nat_{\geq N}} \subset \bar \inspace$, the same converge rate holds for the pointwise asymptotic convergence of $(x_n)_{n\in\nat}$, i.e.
    $$\limsup_{n \to \infty} \mathbb{E}[a_n^{-1}\norm{f(x_{n+1}) - \hat f_n(x_{n+1})}_\outspace ]$$
    with $(a_n)_{n\in\nat} := ((n^{-1}log(n) )^\frac{\alpha}{d+\eta\alpha} )_{n \in \nat}$.
\end{proof}

While Corollary \ref{cor:online conv rate} provides an interesting extension to Theorem \ref{thm:conv rate}, the characterisation of the regularly sampling set $\bar X$ and the necessity of ensuring that Assumption \ref{ass: input space2} holds for $\bar \inspace$ can be difficult to do in practice. Therefore, in comparison to Corollary \ref{cor:online} which can be directly utilised in various control applications, Corollary \ref{cor:online conv rate} is essentially a theoretical result.

\color{black}

\section{Removing the Lipschitz Constant Assumption}

The main difficulty of the Lipschitz interpolation framework is obtaining a suitable hyper-parameter that properly estimates the Lipschitz constant of the unknown target function. In cases where prior knowledge of the Lipschitz constant of $f$ is not obtainable, an additional step is therefore needed. While one solution would be to compute this estimate offline beforehand, this approach is problematic when considering a stream of data. Instead, one can consider the approach developed by \cite{novara2013direct} and applied in the context of Lipschitz interpolation by \cite{calliess2020lazily} which utilises a modified version of Strongin's Lipschitz constant estimator (\cite{strongin1973convergence}) to $(\data_n)_{n \in \nat}$ to obtain a sequence $(L(n))_{n \in \nat}$ of approximations of $L^*$. These estimates can be continuously updated with the arrival of new data and are defined formally in the following definition.
\begin{defn}(LACKI rule)\label{def:LACKI rule}
	The Lazily Adapted Lipschitz Constant Kinky Inference (LACKI) rule computes a Lipschitz interpolation predictor $\predfn$ as per Definition \ref{def:KILsimplified}, but where L depends on $(\data_n)_{n \in \nat}$ and is computed as follows: 
	\begin{equation}
	L(n) := 
	\max \Bigl\{ 0, \max_{(s,s')  \in U_n} \frac{\metric_\outspace(\tilde f(s),\tilde f(s')) - \hestthresh}{\metric(s,s')} \Bigr\}, \label{eq:lazyconstupdaterule_batch_main}
	\end{equation}
	where $U_n = \{(g_1,g_2) \in \grid_n^\inspace \times \grid_n^\inspace | \metric(g_1,g_2) > 0 \}$ and $\lambda$ is a  hyperparameter.
\end{defn}
The noise bounds can be correctly estimated if the $\lambda$ hyper-parameter of the LACKI rule is set to $2 \obserrbnd$. \cite{calliess2020lazily} provides worst-case prediction bounds even when the noise bounds are not correctly estimated. In this paper, we focus on the case where the noise bounds are known and $\lambda$ can be correctly specified. We note that the Lipschitz estimator $L(n)$ given by LACKI is the smallest Lipschitz constant that is consistent with the data. In other words, it reduces the hypothesis space of Lipschitz continuous functions $Lip(L(n),\metric)$ that the target function f could belong to.

We start by showing that the LACKI rule proposed in Definition \ref{def:LACKI rule} converges asymptotically to the best Lipschitz constant of the unknown target function.

\begin{lem} \label{lemma:lacki rule}
If the assumptions of Theorem \ref{lemma:general} hold, then :
$$\forall \epsilon >0, \lim_{n \to \infty} \pmeas(|L(n) - L^*| > \epsilon) = 0$$
\end{lem}
\begin{proof} Fix an arbitrary $\epsilon>0$.
We start by defining an auxiliary function F:
\begin{align*}
     F: Dom(F) := \inspace \times \inspace-\{(x&,x) | x \in \inspace \}  \longrightarrow \Real_{\leq 0} \\
     (x,y) & \longmapsto \frac{\metric_\outspace{(f(x),f(y))}}{\metric(x,y)}
\end{align*}
By construction, $L^* = \sup_{(x,y) \in Dom(F)} F(x,y)$ and there exists $(x_1,x_2) \in Dom(F)$ such that $  L^* - \frac{\epsilon}{2} \leq F(x_1,x_2) \leq L^*$. Hence,
$$ \pmeas \left(|L(n) - L^*| > \epsilon \right) $$ $$\leq \pmeas \left( |L(n) - F(x_1,x_2)|+| F(x_1,x_2) - L^*| > \epsilon \right) $$
$$ = \pmeas \left( |F(x_1,x_2) - L(n)| > \frac{\epsilon}{2} \right) = \pmeas \left( F(x_1,x_2) - L(n) > \frac{\epsilon}{2} \right).$$
 Since F is continuous on its domain, we have that $\exists \delta_1>0$ such that $\forall (x,y) \in B_{\delta_1}((x_1,x_2))$\footnote{Here, $B_\delta((x_1,x_2))$ denotes the ball centered in $(x_1,x_2)$ of radius $\delta_1$ with respect to $\metric_{\inspace\times\inspace}$ defined such that $\metric_{\inspace\times\inspace}((x_1,x_2),(x_1',x_2')) = \metric(x_1,x_1') + \metric(x_2,x_2')$ } $\cap$ $Dom(F)$, $|F(x_1,x_2) -F(x,y)| <\frac{\epsilon}{2}$. Defining  $0< \delta_2 < \min\{ \frac{\delta_1}{2}, \frac{\metric(x_1,x_2)}{2} \}$, we 
 consider the two hyperballs $B_1:= B_{\delta_2}(x_1)$, $B_2:= B_{\delta_2}(x_1)$. Then
 \begin{align*}
  F(x_1,x_2) &- L(n) \\
  = F(x_1,x_2) & - \max_{(s,s')  \in U_n} \frac{|\tilde f(s) - \tilde f(s') | - \hestthresh}{\metric(s,s')} 
 \\
 \leq F(x_1,x_2) & -\max_{s_i\in B_1, s_j\in B_2} \frac{|\tilde f(s_1)-\tilde f(s_j)| - \hestthresh}{\metric(s_i,s_j)}
\\
 \leq F(x_1,x_2)  & -\max_{\substack{s_i\in B_1, s_j\in B_2  \\ cond(s_i,s_j)}} \frac{|f(s_i)- f(s_j)| + |e(s_i)-e(s_j)| - \hestthresh}{\metric(s_i,s_j)}\\
  \leq F(x_1,x_2) & - \min_{\substack{s_i\in B_1, s_j\in B_2  \\ cond(s_i,s_j)}}\frac{|f(s_i) - f(s_j)| }{\metric(s_i,s_j)}  \\
  & -\max_{\substack{s_i\in B_1, s_j\in B_2  \\ cond(s_i,s_j)}} \frac{ |e(s_i)-e(s_j)| - \hestthresh  }{\metric(s_i,s_j) }.
\end{align*}
where  $cond(x,y) := \left\{sgn\left(f(s_i) - f(s_j)\right) = sgn \left(e(s_i)-e(s_j)\right) \right\}$ and with abuse of notation, $\tilde f(s_i)$ $e(s_i)$ denote the noise term associated with the input $s_i$.
By definition of $B_1$, $B_2$, we have 
\begin{align*}
& F(x_1,x_2) - \min_{\substack{s_i\in B_1, s_j\in B_2  \\ cond(s_i,s_j)}}\frac{|f(s_i) - f(s_j)| }{\metric(s_i,s_j)}\\
& = F(x_1,x_2) - \min_{\substack{s_i\in B_1, s_j\in B_2  \\ cond(s_i,s_j)}}F(s_i, s_j) \leq \frac{\epsilon}{4}.
\end{align*}
Substituting this value into the initial expression, we can obtain the upper bound
\begin{align*}
& \frac{\epsilon}{4} -\max_{\substack{s_i\in B_1, s_j\in B_2  \\ cond(s_i,s_j)}} \frac{ |e(s_i)-e(s_j)| - \hestthresh  }{\metric(s_i,s_j) } \\
\leq &
  \frac{\epsilon}{4} + \min_{\substack{s_i\in B_1, s_j\in B_2  \\ cond(s_i,s_j)}} \frac{\hestthresh 
 - |e(s_i) - e(s_j)| }{\metric(s_i, s_j)} \\
 \leq &
\frac{\epsilon}{4} +  \min_{\substack{s_i\in B_1, s_j\in B_2  \\ cond(s_i,s_j)}} \frac{\hestthresh 
 - |e(s_i) - e(s_j)| }{\metric(x_1, x_2) - 2\delta_2 } .
\end{align*}
By the assumption of uniformly dense sampling, there exists $M\in\nat$ such that $r(M) < {\delta_2}$. Therefore, for $n>M$, 
$$\pmeas \left(F(x_1,x_2) - L(n) > \frac{\epsilon}{2} \right)$$
$$
\leq \pmeas \left( \min_{\substack{s_i\in B_1, s_j\in B_2  \\ cond(s_i,s_j)}} \frac{\hestthresh 
 - |e(s_i) - e(s_j)| }{\metric(x_1, x_2) - 2\delta_2 } > \frac{\epsilon}{4} \right)
$$
$$ \leq \pmeas \left( \min_{\substack{s_i\in B_1, s_j\in B_2  \\ cond(s_i,s_j)}} \left\{\hestthresh 
 - |e(s_i) - e(s_j)| \right\} > \frac{\epsilon}{4}(\metric(x_1, x_2)-2\delta_2) \right)  $$ 
$$ =  \pmeas \left(\max_{\substack{s_i\in B_1, s_j\in B_2  \\ cond(s_i,s_j)}} |e(s_i) - e(s_j)|  < \hestthresh -\frac{\epsilon}{4}(\metric(x_1,x_2)-2\delta_2) \right).$$ 
As $\hestthresh = 2\obserrbnd$ and $\metric(x_1,x_2)>2\delta_2$, the last expression can be shown to converge to 0 as $n$ goes to $\infty$ by a similar argument to the one used in the proof of Theorem \ref{lemma:general}.
\end{proof}

Lemma \ref{lemma:lacki rule} proves that the modified version of Strongin's estimate defined in Definition \ref{def:LACKI rule} is a consistent Lipschitz constant estimator under bounded noise. It is therefore of interest for applications outside the one considered in this paper, e.g. see in particular global optimisation methods that depend explicitly on the Lipschitz constant (see for example \cite{malherbe2017global}). One main drawback however is that none of the finite sample estimates generated by the LACKI rule upper bound the true Lipschitz constant. This is discussed in more detail after Theorem \ref{thm:lacki}.


Using Theorem \ref{lemma:general} and Lemma \ref{lemma:lacki rule}, we can now show that the sequence of LACKI predictors $(\hat f_n)_{n \in \nat}$ converges uniformly and in probability to the target function $f$. 

\begin{thm} \label{thm:lacki}
If the assumptions of Theorem \ref{lemma:general} hold, then the sequence of LACKI predictors $(\hat f_n)_{n \in \nat}$ with $\hestthresh = 2\obserrbnd$ converges to f uniformly and in probability:
$$\forall \epsilon >0, \lim_{n \to \infty}\pmeas \left(\sup_{x\in \inspace}\metric_\outspace(\hat f_n(x), f(x)) > \epsilon \right) = 0$$
\end{thm}
\begin{proof}
The proof of Theorem \ref{thm:lacki} follows from Theorem \ref{lemma:general} and Lemma \ref{lemma:lacki rule}.
Fix an arbitrary $\epsilon > 0$, we have
$$ \pmeas \left(\sup_{x\in \inspace}\metric_\outspace(\hat f_n(x), f(x)) > \epsilon \right) $$
$$ \leq \pmeas \left(\sup_{x\in \inspace}\metric_\outspace(\hat f_n(x), \hat f_n^*(x)) > \frac{\epsilon}{2} \right) 
+ 
\pmeas \left(\sup_{x\in \inspace}\metric_\outspace(\hat f_n^*(x),  f(x)) > \frac{\epsilon}{2} \right)$$
where  $(\hat f_n^*)_{n \in \nat}$ denotes the general  Lipschitz interpolation framework with a hyperparameter equal to the best Lipschitz constant $L^*$ of $f$.
The second term of the upper bound given above converges to 0 as $n \to \infty$ by Theorem \ref{lemma:general}. For the first term, we have
$$\pmeas \left(\sup_{x\in \inspace}\metric_\outspace(\hat f_n^*(x), \hat f_n(x)) > \frac{\epsilon}{2} \right)$$
$$
\leq \pmeas \left(\sup_{x \in \inspace}\frac{1}{2}|\min_{i=1,...,N_n} \{ \tilde f_i + L(n)\metric(x,s_i)\} 
-\min_{i=1,...,N_n} \{ \tilde f_i + L^*\metric(x, s_i)\}| > \frac{\epsilon}{4} \right) $$
$$
 + \pmeas \left(\sup_{x \in \inspace}\frac{1}{2}|\max_{i=1,...,N_n} \{ \tilde f_i - L(n)\metric(x,s_i)\}
-\max_{i=1,...,N_n} \{ \tilde f_i - L^*\metric(x,s_i)\}| > \frac{\epsilon}{4} \right) $$
$$ \leq \pmeas \left(\sup_{x \in \inspace}\metric(x,s^*_i)| L^* -L(n)| > \frac{\epsilon}{4} \right)
+ \pmeas \left(\sup_{x \in \inspace}\metric(x, s^*_k)| L^* -L(n)| > \frac{\epsilon}{4} \right) $$
\begin{equation}
    \leq 2\pmeas \left(\delta_{\metric}(\inspace)|L^*-L(n)| > \frac{\epsilon}{4} \right)
\end{equation}  
where $s^*_i:=\argmin_{i=1,...,N_n} \{ \tilde f_i + L(n) \metric(x, s_i)\}$ and $s^*_k:=\argmax_{i=1,...,N_n} \{ \tilde f_i - L(n)\metric(x,s_i)\}$.  As $\delta_{\metric}(\inspace)$ is finite by assumption, Lemma \ref{lemma:lacki rule} can be applied to show that $\pmeas(\delta_{\metric}(\inspace)|L^*-L(n)| > \frac{\epsilon}{4})$ converges to 0.
\end{proof}

In general, it suffices for the sequence of Lipschitz constant estimates to converge to a value that is bigger than the best Lipschitz constant in order for the consistency guarantees given in Theorem \ref{thm:lacki} to hold. This follows from the fact that Lemma \ref{lemma:lacki rule} holds for any Lipschitz interpolation framework with $L\geq L^*$. Furthermore, if the Lipschitz constant estimate can be guaranteed to be feasible\footnote{i.e. $\hat L(n) \geq L^*$.} in a finite number of queries and is asymptotically bounded, then the rate of convergence of the adaptive Lipschitz interpolation method matches the one derived in Theorem \ref{thm:conv rate}. Unfortunately, as remarked above, the LACKI rule proposed in Definition \ref{def:LACKI rule} is not feasible for any finite number of sample points but converges only asymptotically to the true best Lipschitz constant. One approach to remedying this problem would be to include a multiplicative factor $\kappa \geq 1$ (similar to the original approach proposed by \cite{strongin1973convergence} in the noiseless sampling setting) in the LACKI rule. However, developing a principled approach to setting $\kappa$ is non-trivial and depends on second order partial derivatives of the unknown target function. 

Furthermore, in contrast to the general Lipschitz interpolation approach, the LACKI estimator is also not necessarily asymptotically consistent in the setting of a non-linear discrete-time dynamic system. This is due to the fact that dependent on the sampling sequence, the LACKI rule may never become large enough to ensure that the relations (1) and (2) derived in the proof of Lemma \ref{lemma:lacki rule} hold. This issue could potentially be fixed 
by including a "memory hyper-parameter" that limits the number of past observations considered in the $\decke_n$, $\boden_n$ functions. This extension will be investigated in future work.

In essence, while the general Lipschitz interpolation framework can be shown to perform well as a non-parametric estimation method, the additional difficulty of Lipschitz constant estimation implies that many of the desirable asymptotic properties become difficult to obtain for a fully adaptive version of the framework. A detailed discussion on this issue can be found in \cite{huang2023on} where optimal convergence rates are given for the Lipschitz constant estimation problem and a feasible asymptotically consistent estimation method is developed.

\section{Connections to Online Learning and Control} \label{sec: online control}

We conclude this paper by providing a simple illustration of the potential applicability of our results to learning-based control. More precisely, we slightly modify the online consistency results of the general Lipschitz interpolation stated in Section \ref{sec:online} in order to obtain closed-loop stability of a class of online learning-based trajectory tracking controllers discussed in \cite{sanner1991gaussian}, \cite{aastrom2013adaptive}, \cite{chowdhary2013bayesian}, \cite{calliess2020lazily}.

We briefly recall the setting of the trajectory tracking control problem considered by \cite{calliess2020lazily}. The goal is to ensure that a sequence of states $(y_n)_{n\in\nat}$ follows a given reference trajectory $(\xi_n)_{n\in\nat}$. In order to do so, it is assumed that the states $(y_n)_{n\in\nat}$ satisfy a multivariate recurrence relation described as follows: $$y_n = f(x_n) $$
where $x_n = (y_{n-d_y},...,y_{n-1},u_{n-d_u}, ... ,u_{n})$ with $y_i\in\mathcal{Y}\subset\Real^l$ denoting the past autoregressive inputs, $u_i \in\mathcal{U}\subset \Real^s$ 
denoting a vector of past or current control inputs for $d_y, d_u, s, l \in \nat$.  In this setting, we will therefore consider $\inspace = \Real^{(l) (d_y)}\times \mathcal{U}^{d_u+1} \subset \Real^{(l)(d_y) + (s)(d_u+1)}$, $\outspace = \Real^l$. Note that in contrast to the setting considered in Section \ref{sec:online} the noise does not impact the state and will only be assumed to be observational: we assume that the Lipschitz interpolation framework has access to noisy samples of function values $f(x_i)$ at each time step $i<n$: $\data_n =\{(x_i, \tilde f_i) | i<n \}$.


Under this assumption on the system dynamics, the problem becomes equivalent to defining a control law that ensures that the tracking error $(\zeta_n)_{n\in\nat}$, $\zeta_n = \xi_n - y_n$ becomes stable: obtaining, in an ideal scenario, a closed-loop recurrence relation 
$$\zeta_{n+1} = \phi(\zeta_n) $$
where $\phi$ is a contraction with a desirable fixed point $\zeta_*$, typically $\zeta_* = 0$.


This type of stability is well-known to be achievable when the dynamics of the states $(y_n)_{n\in\nat}$ are known and sufficiently well-behaved (\cite{aastrom2013adaptive}) or when $f$ is assumed unknown but well approximated by linear learning-based methods  (\cite{867022}). Obtaining such guarantees in the setting where $f$ is assumed both unknown and non-linear is less straightforward although significant research has been conducted with the use of non-parametric regression methods (\cite{sanner1991gaussian}, \cite{chowdhary2013bayesian}, \cite{calliess2020lazily}).

Under a general assumption on the control law, the online-learning guarantees of the   Lipschitz interpolation method (Corollary \ref{cor:online} and Lemma  \ref{lem:mult online}) derived in this paper can be shown to directly imply the convergence of the tracking error to a fixed point, therefore ensuring the asymptotic stability of the controller.

To do so, we begin by formally extending the online guarantees of the Lipschitz interpolation stated in Corollary \ref{cor:online} to the multi-dimensional online setting described above. In this case, the Lipschitz interpolation framework is applied component-wise as follows:

\begin{defn}\label{def:KILmulti}(Multi-dimensional Lipschitz interpolation) Let $ L \in \Real_{\geq 0}$ be a selected hyper-parameter. Using the set-up defined above, we define the sequence of predictors $( \hat f_n)_{n \in \nat}$, $\hat f_n: \inspace \to \outspace $  associated to $(\data_n)_{n \in \nat}$, as 
   	\begin{equation*}
    \forall j\in\{1, ... ,l \}, \quad
	\predfn^i\bigl(x) := \frac{1}{2} \decke_n^j(x) + \frac{1}{2}  \boden_n^j(x), 
	\end{equation*}
	where $\decke_n^j, \boden_n^j : \inspace \to \Real $ are defined as
	\begin{align*}
	    &\decke_n^j(x) = \min_{i=1,...,N_n} { \tilde f_{n,i}^j} + L \metric(x,s_i)  \\ 
	    &\boden_n^j(x) = \max_{i=1,...,N_n} {\tilde f_{n,i}^j} - L \metric(x,s_i)  \end{align*}
 for all  $j\in\{1, ... ,l \}$.
\end{defn}

We note that under Assumption \ref{ass:outspace norm2} provided below, it is relatively straightforward to observe that each component of the target function is also Lipschitz continuous with the same Lipschitz constant. This implies that the properties utilised in the previous sections hold component-wise for the multi-dimensional Lipschitz interpolation framework.

In order to derive the desired online guarantee for the Lipschitz interpolation framework described in Definition \ref{def:KILmulti}, we first extend the assumptions of the previous sections to the multi-dimensional output setting.

\begin{ass}\label{ass:noise lip2}(Assumption on multi-dimensional noise)
The noise variables $(e_k)_{k \in \nat}$, $e_k \in \Real^d$ are assumed to be independent and identically\footnote{The identically distributed assumption is made to alleviate notation and is not technically needed in our derivations.} distributed random variables such that $\exists \obserrbnd \in \Real_+^d$, such that $\forall k\in\nat$ $j\in\{1,...,d\}$, $e_k^j \in [-\obserrbnd^j, \obserrbnd^j]$ with probability $1$. We assume further that the bounds of the support are tight, i.e.  $\forall \epsilon > 0$, $\forall j\in\{1,...,d\}$, 
$$\pmeas(e_k^j \in [-\obserrbnd^j + \epsilon, \obserrbnd^j]), \pmeas(e_k^j \in [-\obserrbnd^j, \obserrbnd^j -\epsilon]) > 0.$$ 
\end{ass}
\begin{ass} \label{ass:outspace norm2}(Assumption on $\metric_\outspace$).
In this section, we will restrict ourselves to the case, $\metric_\outspace(y,y') = \norm{y-y'}_1$, $\forall y, y' \in \mathcal{Y}$  where $\norm{.}_1$ denotes the usual 1-norm. 
\end{ass}
\begin{rem}
    By the strong equivalence of norms on $\Real^l$, it is sufficient to show the results of this section for $\norm{.}_\outspace = \norm{.}_1$. Additionally, we note that if a Lipschitz constant of the target function is known for a given norm on $\Real^l$, then it is straightforward to compute a feasible Lipschitz constant for any other norm on $\Real^l$.
\end{rem}
\begin{cor}\label{lem:mult online}(Multidimensional Online Learning)
Consider the multidimensional setting described above\footnote{With the same arguments, one can show that the same result holds for the multidimensional version of the dynamical system described in Section \ref{sec:online}.}, $L^*, M^* \in \Real^+$ and $(\hat f_n)_{n\in\nat}$ as defined in Definition \ref{def:KILmulti} with $L \geq L^*$ and $(\mathcal{D}_n)_{n\in\nat} = (x_n, y_n)_{n\in\nat}$. Assume that Assumptions \ref{ass:noise lip2} and \ref{ass:outspace norm2} hold. Assume furthermore that $\mathcal{U}$ is compact. Then
$$
\lim_{n \to \infty} \sup_{f\in \overline{Lip}(L^*, \metric,M^*)} \mathbb{E}[\norm{f(x_{n+1}) - \hat f_n(x_{n+1})}_\outspace] \to 0
$$
where we recall that $\overline{Lip}(L^*, \metric, M^*)$  denotes the set containing all functions in $Lip(L^*, \metric)$ that are bounded by $M^*$, i.e. $\norm{f(x)}_\outspace\leq M^*$.
\end{cor}
\begin{proof}
By the strong equivalence of norms on $\Real^l$, it is sufficient to show Lemma \ref{lem:mult online} for $\norm{.}_\outspace = \norm{.}_1$:
$$
\lim_{n \to \infty} \sup_{f\in \overline{Lip}(L^*, \metric, M^*)} \mathbb{E}[\norm{f(x_{n+1}) - \hat f_n(x_{n+1})}_1] \to 0. 
$$
This is implied if  $\forall j\in\{1,...,l\}$:
$$
\lim_{n \to \infty} \sup_{f\in \overline{Lip}(L^*, \metric, M^*)} \mathbb{E}[ |f^j(x_{n+1}) - \hat f^j_n(x_{n+1})|] \to 0 
$$
where $f^j_n,\hat f^j_n$ denote the j-th component functions of $f_n,\hat f_n$. This statement can be derived using the same arguments as the ones given in the proof of Corollary \ref{cor:online} as, under Assumption \ref{ass:outspace norm2}, $f$ is component-wise Lipschitz continuous with Lipschitz constant $L^*$ and by construction, the multi-dimensional Lipschitz interpolation framework can be considered component-wise.
\end{proof}
Utilising Lemma \ref{lem:mult online}, we can now state the closed-loop guarantees of an online controller based on the Lipschitz interpolation framework.

\begin{thm} \label{thm: online theorem}
    Assume the settings described above. Assume that reference trajectory $(\xi_n)_{n\in\nat}$ is bounded and that the recursive plant dynamics satisfy: $f\in \overline{Lip}(L^*,\metric, M^*)$ for some $L^*, M^* >0$. Let $(\hat f_n)_{n\in\nat}$ be the predictors generated by the Lipschitz interpolation framework with hyperparameter $L\geq L^*$ and $(\mathcal{D}_n)_{n\in\nat} = (x_n, \tilde f_i)_{n\in\nat}$.
    If there exists a bounded control law $u_{n+1} := u(x_n,\hat f_n, \mathcal{D}_n)$ such that the closed loop dynamics are given by:
    $$\zeta_{n+1} = \phi(\zeta_n) + d_n $$
where $d_n:= f(x_n) - \hat f(x_n)$ is the one-step prediction error and $\phi$ is a contraction with a fixed point $\zeta_*\in\inspace$ and Lipschitz constant $\lambda_\phi \in [0,1)$, then we have
$$
\lim_{n\to\infty}\mathbb{E}[\norm{\zeta_n-\zeta_*}_\outspace] = 0.
$$
\end{thm}
\begin{proof}
The proof follows a modified version of the proof Theorem 17 of \cite{calliess2020lazily} and an application of Corollary \ref{thm: online theorem}.

Define the \emph{nominal reference error} $(\bar \zeta_n)_{n\in\nat}$, $\bar \zeta_0 = \zeta_0$, $\bar \zeta_{n+1} = \phi(\bar \zeta_n)$ for $n\in\nat$.
Fix an arbitrary $\epsilon>0$. The proof of Theorem \ref{thm: online theorem} follows from the following sequence of steps.\\
\\
By the Banach fixed point Theorem, $\exists n_0 \in\nat$ such that $\forall n\geq n_0$,
    $\norm{\bar \zeta_n - \zeta_*}_\outspace <\frac{\epsilon}{3}$. 

\noindent Inductively, one can show that $\forall n,k \in \nat$,
    $$\mathbb{E}[\norm{\zeta_{k} - \bar \zeta_{k}}_\outspace] 
    $$ 
    $$
    \leq \lambda_{\phi}^n \mathbb{E}[\norm{\zeta_{k} - \bar \zeta_{k}}_\outspace] + \sum^{n-1}_{i=0} \lambda_\phi^{n-1-i}\mathbb{E}[\norm{d_{k+i}}_\outspace]
    $$
    $$
    \leq \lambda_{\phi}^n \mathbb{E}[\norm{\zeta_{k} - \bar \zeta_{k}}_\outspace] + \frac{1}{1-\lambda_\phi}\max_{i=0,...,n-1}\mathbb{E}[\norm{d_{k+i}}_\outspace].
    $$
    
\noindent By Lemma \ref{lem:mult online}, we have that $\mathbb{E}[\norm{d_{n}}_\outspace]$ converges to 0 as n goes to infinity. Therefore, $\exists k_0 \in \nat$ such that $\forall k\geq k_0$,
    $$
    \frac{1}{1-\lambda_\phi}\max_{i=0,...,n-1}\mathbb{E}[\norm{d_{k+i}}_\outspace] \leq \frac{\epsilon}{3}.
    $$
Let $m_0 := \max\{n_0, k_0 \}$. Since $\lambda_{\phi}^{q_0} < 1$, there exists $q_0\in \nat$ such that:
    $$\lambda_{\phi}^{q_0} \mathbb{E}[\norm{\zeta_{m_0} - \bar \zeta_{m_0}}_\outspace] < \frac{\epsilon}{3}.$$
Let $M := m_0 + q_0$. Combining the above steps, we have
    for any $ m > M$ there exists $n\geq q_0$ such that $m = m_0 + n$. This implies
    $$\mathbb{E}[\norm{\zeta_{m} - \zeta_*}_\outspace] \leq  \norm{\zeta_* - \bar \zeta_{m}}_\outspace + \mathbb{E}[\norm{\zeta_{m} - \bar \zeta_{m}}_\outspace] $$
$$
\leq \frac{\epsilon}{3} + \lambda_{\phi}^{q_0} \mathbb{E}[\norm{\zeta_{m_0} - \bar \zeta_{m_0}}_\outspace] +  \frac{1}{1-\lambda_\phi}\max_{i=0,...,n-1}\mathbb{E}[\norm{d_{m_0+i}}_\outspace] 
$$
$$
\leq \frac{\epsilon}{3} + \frac{\epsilon}{3} + \frac{\epsilon}{3} = \epsilon.
$$   
As the choice of $\epsilon$ was arbitrary, this concludes the proof.
\end{proof}

Theorem \ref{thm: online theorem} provides a theoretical result on the global stability in expectation of a general class of control problems where both the system dynamics and the error dynamics are assumed to be both unknown and non-linear. An extension of Theorem \ref{thm: online theorem} which states that the convergence rates of the Lipschitz constant estimator derived in Corollary \ref{cor:online conv rate} hold for the tracking error $(\zeta_n)_{n\in\nat}$ can also be obtained. However, this result would be contingent on the difficult-to-verify "regularity of the sampling" assumption of $(x_n)_{n\in\nat}$ (as defined in Definition \ref{def:reg samp}) and is therefore of limited interest. We provide it in the appendix for completeness.

Unfortunately, for numerous applications, the contraction assumption on dynamics of the tracking error $\phi$ is too stringent to be achieved in practice. To alleviate this issue, Theorem \ref{thm: online theorem} can be extended to consider the more general assumption that $\phi$ is an eventually contracting function if $\phi$ is also assumed to be a linear.  More formally, we define an eventually contracting function as follows.
\begin{defn}(Eventually Contracting Function)
    Let $l\in\nat$. A continuous function $h:\Real^l \to \Real^l$ is said to be eventually contracting if there exists $N \in \nat$ and $\lambda \in [0,1)$ such that $\forall x,y \in \Real^l $:
    $$\metric(h^N(x), h^N(y) ) \leq \lambda \metric(x,y).$$
\end{defn}
As with the contracting functions considered above, eventually contracting functions can be shown to admit a unique fixed point $\xi^*$. Additionally, it is well-known that a linear function $h: \Real^l \to \Real^l$ defined as $h(x) = Mx$ for some matrix $M\in\Real^{l \times l}$ is eventually contracting if and only if the spectral radius of $M$ is strictly smaller than 1: $\rho(M)<1$. 

The assumption of the existence of a control law $u_{n+1} := u(x_n,\hat f_n, \mathcal{D}_n)$ such that the closed loop dynamics are given by: $\zeta_{n+1} = \phi(\zeta_n) + d_n $and  $\phi$ is eventually contracting can be observed in applications such as the removal of wing rock during the landing of modern fighter aircraft (\cite{monahemi1996control}, \cite{chowdhary2013bayesian}). Therefore, if Theorem \ref{thm: online theorem} can be extended to hold under these assumptions, then, conditional on the existence of feasible Lipschitz constant estimate, the online learning-based reference tracking controllers utilising a Lipschitz interpolation can be ensured to be globally asymptotically stable in expectation in these settings. 
This alternative result is stated in the following corollary.

\begin{cor} \label{cor: online theorem}
    Assume the setting and initial assumptions of Theorem \ref{thm: online theorem}. 
    If there exists a bounded control law $u_{n+1} := u(x_n,\hat f_n, \mathcal{D}_n)$ such that the closed loop dynamics are given by:
    $$\zeta_{n+1} = \phi(\zeta_n) + d_n $$
where $d_n:= f(x_n) - \hat f(x_n)$ is the one-step prediction error and $\phi:\Real^l \to \Real^l$, $\phi(\zeta) = M\zeta$ for a matrix $M\in\Real^{l \times l}$ that is a stable, i.e. $\rho(M) < 1$. Then
$$
\lim_{n\to\infty}\mathbb{E}[\norm{\zeta_n}_\mathcal{Y}] = 0.
$$
\end{cor}
\begin{proof}
    The proof of Corollary \ref{cor: online theorem} is similar to the on given for Theorem \ref{thm: online theorem}.

Define the nominal reference error $(\bar \zeta_n)_{n\in\nat}$, $\bar \zeta_0 = \zeta_0$, $\bar \zeta_{n+1} = \phi(\bar \zeta_n)$ for $n\in\nat$. Fix an arbitrary $\epsilon>0$. 

As $\rho(M) < 1$, we have that  $\lim_{n\to\infty} \bar \zeta_n = 0$ (\cite{hasselblatt2003first}). This implies that  $\exists n_0 \in\nat$ such that $\forall n\geq n_0$,
    $\norm{\bar \zeta_n}_\outspace <\frac{\epsilon}{3}$. 

Inductively, one can show that $\forall n,k \in \nat$,
    $$
    \mathbb{E}[\norm{\zeta_{n+k} - \bar \zeta_{n+k}}_\outspace] 
    $$ 
    $$
    \leq  \norm{ M^n}_\outspace \mathbb{E} \left[\norm{ (\zeta_{k} - \bar \zeta_{k} )}_\outspace \right] + \sum^{n-1}_{i=0} \norm{M^{n-1-i}}_\mathcal{Y} \mathbb{E}[\norm{d_{k+i}}_\outspace].
    $$
By Gelfand's formula we have $\lim_{n \to \infty} \norm{M^k}^{\frac{1}{k}} = \rho(M) < 1$ for any matrix norm $\norm{.}$. This implies that there exists $n_1$ such that for all $n \geq n_1$: $\norm{M^n}_\outspace \leq \norm{M^n}_\outspace^{\frac{1}{n}} < \lambda_\phi < 1$ for some $\lambda_\phi \in (0,1)$. Utilising this relation and matrix-vector inequalities, we obtain the following inequalities: let $n\geq n_1$, there exists $n_2 \in \nat \cup \{0 \}$ such that $n = n_1 \floor{\frac{n}{n_1}} + n_2$:
$$
\norm{ M^{n}}_\outspace  = \norm{ M^{ n_1 (\floor{\frac{n}{n_1}} - 1) } M^{n_1 + n_2} }_\outspace \leq \norm{ M^{ n_1 }}_\outspace^{(\floor{\frac{n}{n_1}} - 1) } \norm{M^{n_1 + n_2} }_\outspace \leq \lambda_\phi^{(\floor{\frac{n}{n_1}}-1)} \lambda_\phi = \lambda_\phi^{\floor{\frac{n}{n_1}}}.
$$
Substituting this inequality into the bound given above, we obtain:
$$
 \norm{ M^{n}}_\outspace \mathbb{E} \left[\norm{ (\zeta_{k} - \bar \zeta_{k} )}_\outspace \right] + \sum^{n-1}_{i=0} \norm{M^{n-1-i}}_\mathcal{Y} \mathbb{E}[\norm{d_{k+i}}_\outspace]
$$
 $$
    \leq  \lambda_\phi^{\floor{\frac{n}{n_1}}} \mathbb{E}[\norm{ (\zeta_{k} - \bar \zeta_{k} )}_\outspace] + \max_{i=0,...,n-1}\mathbb{E}[\norm{d_{k+i}}_\outspace] \left(K_{n_1} + \sum^{\ceil{\frac{n-1}{n_1}}}_{i=1} n_1 \lambda_\phi^i \right)
    $$
$$
    \leq  \lambda_\phi^{\floor{\frac{n}{n_1}}} \mathbb{E}[\norm{ (\zeta_{k} - \bar \zeta_{k} )}_\outspace] + \max_{i=0,...,n-1}\mathbb{E}[\norm{d_{k+i}}_\outspace] \left(K_{n_1} + \frac{n_1}{1- \lambda_\phi} \right)
    $$
\noindent where $K_{n_1} := \sum^{n_1-1}_{i=0}  \norm{M^{i}}_\mathcal{Y} $
By Lemma \ref{lem:mult online}, we have that $\mathbb{E}[\norm{d_{n}}_\outspace]$ converges to 0 as n goes to infinity. Therefore, $\exists k_0 \in \nat$ such that $\forall k \geq k_0$,
    $$
    \left(K_{n_1} + \frac{n_1}{1- \lambda_\phi} \right)\max_{i=0,...,n-1}\mathbb{E}[\norm{d_{k+i}}_\outspace] \leq \frac{\epsilon}{3}.
    $$
Let $m_0 := \max\{n_0, k_0 \}$. There exists $q_0\in \nat$ such that
    $$\lambda_{\phi}^{\floor {\frac{q_0}{n_1}}} \mathbb{E}[\norm{\zeta_{m_0} - \bar \zeta_{m_0}}_\outspace] < \frac{\epsilon}{3}.$$
Let $M := m_0 + q_0$. Combining the above steps, we have that
    for all $ m > M$, there exists $n\geq q_0$ such that $m = m_0 + n$. This implies
    $$\mathbb{E}[\norm{\zeta_{m}}_\outspace] \leq  \norm{ \bar \zeta_{m}}_\outspace + \mathbb{E}[\norm{\zeta_{m} - \bar \zeta_{m}}_\outspace] $$
$$
\leq \frac{\epsilon}{3} + \lambda_{\phi}^{\floor{\frac{q_0}{n_1}}} \mathbb{E}[\norm{\zeta_{m_0} - \bar \zeta_{m_0}}_\outspace] +  \left(K_{n_1} + \frac{n_1}{1- \lambda_\phi} \right)\max_{i=0,...,n-1}\mathbb{E}[\norm{d_{m_0+i}}_\outspace] 
$$
$$
\leq \frac{\epsilon}{3} + \frac{\epsilon}{3} + \frac{\epsilon}{3} = \epsilon.
$$   
As the choice of $\epsilon$ was arbitrary, this concludes the proof.
\end{proof}

 \subsection{Example - model-reference adaptive control of a single pendulum}
As a simple illustration, we replicate a modification of the model-reference adaptive control example in \cite{calliess2020lazily}. Here, we control an Euler discretisation of a torque-actuated single pendulum in set-point control mode: 

We consider a torque controlled pendulum where forces $u$ can be applied to the joint of the pendulum. The angle of the pendulum is called $q$. We define a state $x=[q\dot q]$. In continuous-time, it's dynamics are given by the ODE $\ddot q = f(x) +u$ where $f(x) = - \sin (q) - \dot q$ may be uncertain a priori and hence, needs to be learned online while we control. 

\begin{figure*}[t]
\centering
\includegraphics[width=0.32\textwidth]{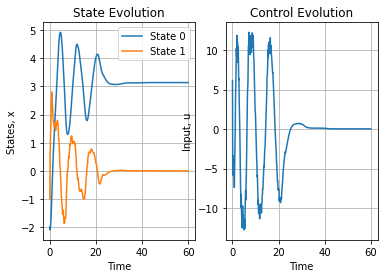}  
\includegraphics[width=0.32\textwidth]{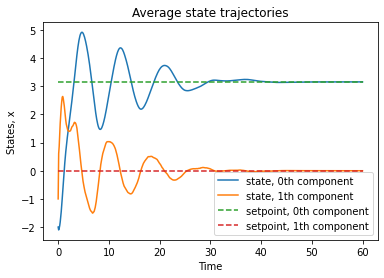} 
	\includegraphics[width=0.32\textwidth]{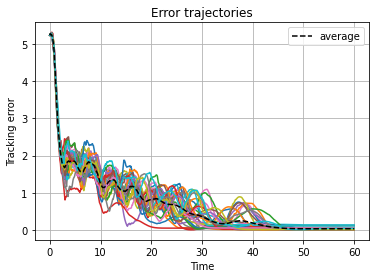}%
	\caption[Illustration of the pendulum control example.]{ 
 Illustration of the pendulum control example. A single run is depicted in the leftmost figure showing how the controller learns to drive the state to the set-point in spite of the noise and initially uncertain dynamics. Note, ``Time'' is simulation time $t = \tinc n$ [sec.] for discrete time steps $n=0,1,...$. The second figure shows how the mean trajectories, averaged over 30 repetitions (each with new draws from the noise distribution), converge to the set-point. An illustration of our theory, predicting vanishing tracking errors in the mean, is depicted in the rightmost figure: For each repetition of the experiment, the colored lines show the error trajectories $ (\norm{\xi({\tinc n}) - x(\tinc n)})_{n=0,1,2,\dots}$ as well as their empirical mean (black dashed line). }	\label{Fig:pend_ex}%
\end{figure*}
As explained in Section 4 of \cite{calliess2020lazily}, using online learning of noisy measurements of angular accelerations (but assuming full state observability) we can use Lipschitz interpolation to learn a model $\hat f_n$ at time step $n$ define a control law $u(x) = - \hat f_n(x) - K_1 x_1 - K_2 x_2$ with gains $K_1,K_2 > 0$ such that the closed-loop error dynamics becomes:
\begin{align}\label{eq:errordynmrac_nIMRACdiscrete}
	\zeta_{n+1}  &= \phi(\zeta_n) + \tinc d_n\\
 &= \underbrace{M \zeta_n}_{=: \phi(\zeta_n)} + \tinc d_n
\end{align}
where $\tinc =0.1$ is a sampling period for the time discretisation, $d_n = f(x_n) - \hat f_n(x_n)$ is the tracking error and $	M = \left(\begin{array}[h]{cc}
				1 &  \, \tinc \\
				-\tinc K_1 & 1- \tinc K_2 
						\end{array}\right) $ is a matrix, where the gain parameters $K_1,K_2=1$ were chosen to render $M$ stable (i.e. such that its spectral radius $\rho(M) <1$).    
This renders the closed-loop tracking error dynamics consistent with the one considered in Corollary \ref{cor: online theorem} and as previously discussed, implies that the error dynamics are an eventual contraction. 

For this experiment, we chose Lipschitz interpolation with a fixed Lipschitz constant $L=11$ which is a Lipschitz constant of the true dynamics. To learn $f$ the Lipschitz interpolator had access to acceleration corrupted by uniformly distributed noise drawn i.i.d. from the interval [-2,2] given a performance example in a relatively low signal-to-noise ration setting. Starting in initial state $x_0= [-2., -1.]$, the controller was given a set-point reference $\xi_n = [2 \pi,0], \forall n \in \nat$. An example run of the controller as well as empirical measurements of the tracking dynamics across 30 trials of the experiments for different noise realisations are given in Figure \ref{Fig:pend_ex}. Note, consistent with our theory, the plots show how the tracking error appears to vanish in the mean (and in fact for all realisations).

Before we conclude, we will need to point out some limitations to the online control application given in this section:
\begin{rem}
Firstly, our example assumed knowledge of the Lipschitz constant. While not an uncommon assumption, we recognise that having to know the true Lipschitz constant is a practical limitation. Therefore methods such as LACKI (\cite{calliess2020lazily}) or POKI (\cite{calliess2017lipschitz}) could also be employed to incorporate full black-box learning. Our example however, merely serves as a simple illustration of our currently developed theory. 
Secondly, we assumed that states were observable but that only accelerations were noisy. 

Extending our results to learning-based control settings that involve Lipschitz constant parameter estimation and extending the theory to realistic settings of noisy state observations would, in our opinion, be an interesting direction to investigate in future work.
\end{rem}

\section{Conclusion}

In conclusion, this paper has provided a comprehensive investigation into the asymptotic convergence properties of general Lipschitz interpolation methods in the presence of bounded stochastic noise. Through our analysis, we have established probabilistic consistency guarantees for the classical approach within a broad context. Furthermore, by deriving upper bounds on the uniform convergence rates, we have aligned these bounds with the well-established optimal rates of non-parametric regression observed in similar settings. These rates provide a precise characterisation of the impact of the behaviour of the noise at the boundary of its support on the non-parametric uniform convergence rates  and are, as far as the authors of this paper are aware, novel to the literature.

These established bounds can also serve as useful tools for the comparative asymptotic assessment of Lipschitz interpolation against alternative non-parametric regression techniques determining the circumstances under which Lipschitz interpolation frameworks can be anticipated to be asymptotically better or worse. In particular,  an explicit condition on the noise's behaviour at the boundary of its support can be utilised to predict this out-or under-performance.

Extending our work, we have expanded our asymptotic results to consider online learning in discrete-time stochastic systems. The additional consistency guarantees we provide in this context carry practical significance, as we show how they can be utilised to establish closed-loop stability assurances for a simple online learning-based controller in the setting of model reference adaptive control. We note that these asymptotic results also hold for the worst-case upper and lower bounds provided by the ceiling and floor predictors of the Lipschitz interpolation framework. This implies that even the most conservative adaptive controllers relying on worst-case bounds of Lipschitz interpolation methods will consider the true underlying dynamics in the long run.

Finally, we have provided a brief theoretical study of the fully data-driven LACKI framework (\cite{calliess2020lazily}) which extends classical Lipschitz interpolation by incorporating a Lipschitz constant estimation mechanism into the algorithm. We show asymptotic consistency of both the Lipschiz constant estimation method and the extended framework which can serve to further theoretically motivate the use of the LACKI  in practice.

Two research avenues are of interest with respect to the theoretical results derived in this paper. The first would be the derivation of lower bounds on the non-parametric convergence rates under the same settings and assumptions as the ones utilised in Theorem \ref{thm:conv rate}. These would, ideally but not unexpectedly, given existing results on the optimal convergence rates of non-parametric boundary regression by \cite{jirak2014adaptive}, demonstrate the optimality of the upper bounds on the non-parametric convergence rates developed in this paper. The second potential research direction would be to extend the convergence rate upper bounds stated in Theorem \ref{thm:conv rate} such that they hold for the practical and fully data-driven extensions of Lipschitz interpolation such as LACKI (\cite{calliess2020lazily}) or POKI \cite{calliess2017lipschitz}. 

\bibliography{bibliography}

\appendix

\section{Additional Results (Convergence rate of tracking error)}

We provide the theoretical convergence rates obtained for the tracking error in the application to online learning-based control. As noted in Section \ref{sec: online control}, this result is not generally applicable as verification of the "regular sampling" condition defined in Definition \ref{def:reg samp} is difficult to do in practice.

\begin{cor}\label{cor:multi online} 
Assume that the setting and assumptions of Corollary \ref{lem:mult online} hold. Assume furthermore
that the stochastic control law $u_{n+1} := u(x_n,\hat f_n, \mathcal{D}_n)$ is defined such that $(x_n)_{n\in\nat}$ is regularly sampled on a set $\bar \inspace \subset \inspace$ that satisfies Assumption \ref{ass: input space2} and that the noise vectors $(\epsilon_n)_{n\in\nat}$ are component-wise independent. Then,
$$ 
\limsup_{n \to \infty} \mathbb{E}[a_n^{-1}\norm{f(x_{n+1}) - \hat f_n(x_{n+1})}_\outspace ] < \infty.
$$
where $(a_n)_{n\in\nat} := ((n^{-1}log(n) )^\frac{\alpha}{d+\eta\alpha} )_{n \in \nat}$.
\end{cor}
\begin{proof}(sketch)
    Applying the same arguments as in the proof of Lemma \ref{lem:mult online}, it is sufficient to consider: $\forall j \in \{1,...,l\}$
    $$
\limsup_{n \to \infty} \mathbb{E}[ a_n^{-1}|f^j(x_{n+1}) - \hat f^j_n(x_{n+1})|].
$$
Since the noise is component-wise independent, we can apply the arguments utilised in the proof of Corollary \ref{cor:online conv rate} for all $j\in \{1,...,l\}$ and conclude the proof.
\end{proof}

\begin{thm} \label{thm: online theorem2}
    Assume the settings and assumptions of Theorem \ref{thm: online theorem} hold.
    If the stochastic control law $u_{n+1} := u(x_n,\hat f_n, \mathcal{D}_n)$ is such that $(x_n)_{n\in\nat}$ is regularly sampled on a set $\bar \inspace \subset \inspace$ that satisfies Assumption \ref{ass: input space2} and the noise vectors $(\epsilon_n)_{n\in\nat}$ are component-wise independent. Then,
$$
\limsup_{n\to\infty}\mathbb{E}[a_n^{-1}\norm{e_n-e_*}_\outspace] < \infty
$$
where $(a_n)_{n\in\nat} := ((n^{-1}log(n) )^\frac{\alpha}{d+\eta\alpha} )_{n \in \nat}$.
\end{thm}
\begin{proof}(sketch)
    Follows from applying the proof of Theorem \ref{thm: online theorem} and noting that the slowest converging term at the end of the proof is given by
    $$\frac{1}{1-\lambda_\phi}\max_{i=0,...,n-1}\mathbb{E}[\norm{d_{m_0+i}}_\outspace].$$
    This term can be upper bounded by applying Corollary \ref{cor:multi online} which therefore provides the convergence rate and concludes the proof.
\end{proof}

\section{Proof of Theorem \ref{thm:conv rate}} \label{sec: proof of thm: conv rate}
\begin{proof}
From the proof of Theorem \ref{lemma:general}, we have $\forall f \in Lip(L^*, \metric)$, $x\in \inspace$, 
$$
 |\hat f_n(x) - f(x)|
$$
$$
\leq \max \Big\{ \min_{i=1,...,N_n} \{ \frac{e_i}{2} + \frac{L^*+L}{2}\metric(x,s_i)\} + \max_{i=1,...,N_n} \{\frac{e_i}{2}  \} , 
$$
$$
- \min_{i=1,...,N_n} \{\frac{e_i}{2} \} - \max\limits_{i=1,...,N_n} \{ \frac{e_i}{2} - \frac{L^*+L}{2}\metric(x,s_i)\}  \Big\} .
$$
Consider the minimal covering of $\inspace$ of radius $R_n := a_n = \left(n^{-1}log(n)\right)^\frac{\alpha}{d+\eta \alpha}$ with respect to $\metric$ denoted $Cov(R_n)$ and the associated set of hyperballs; $\mathcal{B}_n$. Assuming that $n$ is large enough such that every hyperball in $\mathcal{B}_n$ contains at least one input of $G_n^\inspace$, we have that the following upper bound holds:
$$
|\hat f_n(x) - f(x)| 
$$
$$
\leq \max \left\{ \min_{s_i \in  B^x \cap  G_n^\inspace} \{ \frac{e(s_i)}{2}\},
\min_{s_i \in B^x \cap  G_n^\inspace} \{-\frac{e(s_i)}{2} \}
\right \} + \frac{\obserrbnd}{2}+ (L^*+L) R_n  
$$ 
where with abuse of notation, $e(s_i)$ denotes the noise variable associated with the input $s_i$ and $B^x  \in \mathcal{B}_n$ such that $x \in B$.
For all $n\in\nat$, we define the following random variable and event
\begin{align*}
A_n & := \max_{B \in \mathcal{B}_n} \max \left\{ \min_{s_i \in  B \cap  G_n^\inspace} \{ \frac{e(s_i)}{2} \},
 - \max_{s_i \in B \cap  G_n^\inspace} \{ \frac{e(s_i)}{2} \} \right\}+ \frac{\obserrbnd}{2} \\ 
 E_n & := \left\{ \forall B \in \mathcal{B}_n, |\{i\in [n] \}, s_i \in B\}| > 0   \right\}.
\end{align*}
Then, from ($\star$) given at the end of the proof we have that it is sufficient to consider 
$$\sup_{f \in Lip(L^*, \metric)} \mathbb{E}\left[a_n^{-1} \norm{\hat f_n - f}_\infty \Big| E_n \right] \pmeas(E_n)
$$ in order for Theorem \ref{thm:conv rate} to hold. For $n \in \nat$ sufficiently large such that $ \pmeas(E_n) > 0$ (see ($\star$)), we can apply the upper bound on $|\hat f_n(x) - f(x)|$ derived above:
$$ \sup_{f \in Lip(L^*, \metric)} \mathbb{E} \left[a_n^{-1}\norm{\hat f_n - f}_\infty \Big| E_n \right] 
= 
\sup_{f \in Lip(L^*, \metric)} \mathbb{E}\left[a_n^{-1}\sup_{x\in\inspace}|f_n(x) - f(x) | \Big| E_n \right] 
$$
$$\leq a_n^{-1} \sup_{f \in Lip(L^*, \metric)} \mathbb{E} \left[\sup_{x\in\inspace}
\max \left \{ \min_{s_i \in  B^x \cap  G_n^\inspace} \{ \frac{e(s_i)}{2}\}
,
\min_{s_i \in B^x \cap  G_n^\inspace} \{-\frac{e(s_i)}{2} \}
\right \} + \frac{\obserrbnd}{2}+ (L^*+L) R_n \Big| E_n \right] 
$$
$$= a_n^{-1} (L^*+L) R_n  
+
\mathbb{E} \left[a_n^{-1}\max_{B\in\mathcal{B}_n}
\max \left\{ \min_{s_i \in  B \cap  G_n^\inspace} \{ \frac{e(s_i)}{2}\},
\min_{s_i \in B \cap  G_n^\inspace} \{-\frac{e(s_i)}{2} \}
\right\} + \frac{\obserrbnd}{2}  \Big| E_n \right] 
$$
$$
= (L^*+L)a_n^{-1}R_n + \mathbb{E}\left[a_n^{-1}A_n \Big| E_n \right]. 
$$
By definition of $R_n$, the first term: $(L^*+L) a_n^{-1}R_n = (L^*+L)$ is bounded for all $n \in \nat$. We can therefore focus on upper bounding the second term:
$$
\mathbb{E} \left[a_n^{-1}A_n \Big| E_n \right]\pmeas(E_n) = \mathbb{E}\left[a_n^{-1}A_n \right].$$ 
Using $0 \leq A_n \leq 2\obserrbnd$ with probability 1, we have $\forall C_0>0$,
$$a_n^{-1}A_n \leq  C_0 1_{\{a_n^{-1}A_n \leq C_0 \}} + 2\obserrbnd a_n^{-1} 1_{\{a_n^{-1}A_n > C_0 \}}$$
with probability $1$. This implies that
$$\mathbb{E}[a_n^{-1}A_n] \leq
C_0 + 2\obserrbnd a_n^{-1} \pmeas(A_n >C_0 a_n).
$$
It is therefore sufficient to show that $\exists C_0>0$ such that 
$$\limsup_{n \to \infty} \sup_{f \in Lip(L^*,\metric)}a_n^{-1}\pmeas(A_n >C_0 a_n) < \infty.$$
We have
$$
\pmeas(A_n >C_0 a_n)
= 
1 - \pmeas ( \forall B \in \mathcal{B}_n:  \min_{s_i \in  B \cap  G_n^\inspace} e(s_i) \in I_1, 
\max_{s_i \in B \cap  G_n^\inspace} e(s_i) \in I_2 )
$$
$$
\stackrel{(\star\star)}{\leq} 1 - \prod_{ B \in \mathcal{B}_n}\pmeas \left(  \min_{s_i \in  B \cap  G_n^\inspace} e(s_i) \in I_1, 
\max_{s_i \in B \cap  G_n^\inspace} e(s_i) \in I_2 \right)
$$
$$
\leq 
1 - \pmeas \left(  \min_{i \in 1,..., N_{\mathcal{B}_n}} e_i \in I_1, 
\max_{i \in 1,..., N_{\mathcal{B}_n}} e_i \in I_2 \right )^{|\mathcal{B}_n|}
$$
where $I_1 := [-\obserrbnd, -\obserrbnd + 2C_0 a_n )$, $I_2 := (\obserrbnd - 2C_0 a_n, \obserrbnd  ]$ and $N_{\mathcal{B}_n} := \min_{B \in \mathcal{B}_n} |B \cap G_n^\inspace |$. The second to last inequality follows from arguments given in $(\star \star)$ provided at the end of the proof.  
For $n$ large enough such that $2C_0 a_n <\bar \epsilon$, we can apply Assumption \ref{ass:noise lip} to simplify the left hand expression as follows;
$$\pmeas \left(  \min_{i \in 1,..., N_{\mathcal{B}_n}} e_i \in I_1, 
\max_{i \in 1,..., N_{\mathcal{B}_n}} e_i \in I_2 \right)
$$
$$
\geq  \pmeas \left(  \min_{i \in 1,..., N_{\mathcal{B}_n}} e_i \in I_1 \right)
\cdot \pmeas \left(
\max_{i \in 1,..., N_{\mathcal{B}_n}} e_i \in I_2 | \min_{i \in 1,..., N_{\mathcal{B}_n}} e_i \in I_1 \right) 
$$
$$
\geq \left(1  -  \left(1 - \gamma (2 C_0 a_n)^\eta \right)^{N_{\mathcal{B}_n}} \right) \left(1  -  \left(1 - \gamma (2 C_0 a_n)^\eta \right) ^{N_{\mathcal{B}_n} -1} \right)
$$
$$
\geq \left(1  -  2 \left(1 - \gamma (2 C_0 a_n)^\eta \right)^{N_{\mathcal{B}_n}-1} \right)^2.
$$
Therefore, we have 
$$
\sup_{f \in Lip(L^*,\metric)}a_n^{-1}\pmeas(A_n >C_0 a_n)
$$
$$
\leq a_n^{-1} \left(1 - \left(1  -  2(1 - \gamma (2 C_0 a_n)^\eta )^{N_{\mathcal{B}_n}-1} \right)^{2|\mathcal{B}_n|} \right) .
$$
$$
\leq a_n^{-1} \left(1 - \left(1  -  2(1 - \gamma (2 C_0 a_n)^\eta )^{N_{\mathcal{B}_n}-1} \right)^{\frac{C_1}{{a_n}^{\frac{d}{\alpha}}}} \right).
$$
where we used the fact that there exists a constant $ C_1 > 0$ (that can depend on $d$) such that $2 |\mathcal{B}_n| \leq \frac{ C_1}{{R_n}^{\frac{d}{\alpha}}} = \frac{C_1}{{a_n}^{\frac{d}{\alpha}}}$ which is a modification of (\cite{wu2017lecture}, Theorem 14.2) that follows from the assumed convexity of $\inspace$. By Lemma \ref{lem:sequence proof},
in order for the above expression to be bounded, it is sufficient that
$2(1 - \gamma (2 C_0 a_n)^\eta)^{N_{\mathcal{B}_n}-1}$ behaves like $C_2' {a_n}^{({\frac{d}{\alpha}}+1)}$ for an arbitrary $C_2'>0$ as $n$ goes to infinity. More precisely, let $C_2' = 1$, it is sufficient to have:
$$
2(1 - \gamma (2 C_0 a_n)^\eta)^{N_{\mathcal{B}_n}-1} \leq {a_n}^{({\frac{d}{\alpha}}+1)}$$
$$
\iff 
N_{\mathcal{B}_n} \geq 1+ ({\frac{d}{\alpha}}+1)\frac{\log \left(  a_n \right)}{\log \left(1- \gamma (2 C_0 a_n)^\eta \right)} 
$$
as $n$ goes to infinity.
The right-hand expression can be re-expressed as the series expansion: 
$$\frac{d + \alpha }{\alpha \gamma (2 C_0)^\eta} \frac{1}{{a_n}^\eta} \log (\frac{1}{a_n})  + O(\log(\frac{1}{ a_n}))$$ 
as $a_n$ goes to $0$. Therefore, for any $C_2 > \frac{d + \alpha }{\alpha \gamma(2 C_0)^\eta}$ and $n>0$ large enough, we have $1+ ({\frac{d}{\alpha}}+1)\frac{\log(a_n)}{\log(1- \gamma (2 C_0 a_n)^\eta )} < C_2 \log(\frac{1}{a_n})\frac{1}{{a_n}^\eta}$ and it suffices to have
$$ 
N_{\mathcal{B}_n} \geq C_2 \log(\frac{1}{a_n})\frac{1}{{a_n}^\eta} 
$$
as $n$ goes to infinity in order for $\lim_{n\to\infty}\sup_{f \in Lip(L^*,\metric)}a_n^{-1}\pmeas(A_n >C_0 a_n)$ to be bounded: 
\newline
If $N_{\mathcal{B}_n} \geq C_2 \log(\frac{1}{a_n})\frac{1}{{a_n}^\eta}$, 
$$
\limsup_{n \to \infty} \sup_{f \in Lip(L^*,\metric)}a_n^{-1}\pmeas(A_n >C_0 a_n)
\leq a_n^{-1} \left(1 - \left(1  -  {a_n}^{({\frac{d}{\alpha}}+1)} \right)^{\frac{C_1}{{a_n}^{\frac{d}{\alpha}}}} \right) \leq 2 C_1
$$
where the last inequality follows from Lemma \ref{lem:sequence proof}.

Therefore, the final step of the proof is to show that $ 
N_{\mathcal{B}_n} \geq C_2 \log(\frac{1}{a_n})\frac{1}{{a_n}^\eta} 
$ occurs with a probability that converges to $1$ at a rate of $a_n$ as n goes to infinity. 

More precisely, let $ n \in \nat$ and fix an arbitrary constant $C_2 > \frac{d + \alpha }{\alpha \gamma(2 C_0)^\eta}$ based on the condition given above (note that $C_0>0$ can be set arbitrarily large).
$$
S_n := \left\{\forall B \in \mathcal{B}_n: \text{ } |\{i\in [n] \}, s_i \in B\}| > C_2 \log(\frac{1}{a_n})\frac{1}{{a_n}^\eta} \right\}
$$ i.e. the event that there is more than $C_2 \log(\frac{1}{a_n})\frac{1}{{a_n}^\eta}$ queries in each hyperball in $\mathcal{B}_n$. Utilising the asymptotic bound developed in the first part of the proof, we 
have by the law of total probability: 
$$
\limsup_{n \to \infty} \mathbb{E}[a_n^{-1}A_n ] 
\leq  
C_0 + 2\obserrbnd \limsup_{n \to \infty}  a_n^{-1} \pmeas(A_n >C_0 a_n)
$$
$$\leq C_0 + 2\obserrbnd\limsup_{n \to \infty} a_n^{-1} \left(\pmeas(A_n >C_0 a_n |S_n)+ \pmeas(S_n^c) \right)
\leq 
C_0 + 4\obserrbnd C_1 + \limsup_{n \to \infty} a_n^{-1}  \pmeas(S_n^c)
$$
where the last inequality can be obtained by applying Lemma \ref{lem:sequence proof}.

To conclude the proof, we need to show that $\limsup_{n \to \infty} a_n^{-1}  \pmeas(S_n^c)$ is bounded. We have (denoting $b_n := \log(\frac{1}{a_n})\frac{1}{{a_n}^\eta}$ to alleviate notation):
$$ 
\pmeas(S_n) = \pmeas \left(\forall B \in \mathcal{B}_n:|\{i\in [n] \}, s_i \in B\}| > C_2 b_n \right)
$$
$$ 
\geq \prod_{B \in \mathcal{B}_n}\pmeas \left(|\{i\in [\floor{\frac{n}{2}}] \}, s_i \in B\}| > C_2 b_n \right)
$$
$$
\geq 
\prod_{B \in \mathcal{B}_n} \left(1 - \pmeas \left(|\{i\in [\floor{\frac{n}{2}}] \}, s_i \in B\}| \leq C_2 b_n \right) \right)
$$
%
where the first inequality stated above follows from $(\star\star\star)$ (shown at the end of the proof) for $n\in\nat$ if $C_2$ satisfies the condition: $C_2 \leq \frac{d+ \eta \alpha}{ 2 C_1 \alpha }$ where $C_1, d, \alpha, \eta$ are constants. 

Then, Assumption \ref{ass: input space2} on $\inspace$ implies that there $r_0>0, \theta \in(0,1]$ such that  $\forall x \in \mathcal{X}$, $r \in\left(0, r_0\right), \operatorname{vol}\left(B_r(x) \cap \mathcal{X}\right) \geq \theta \operatorname{vol}\left(B_r(x)\right)$. Therefore, for all $n\in\nat$ such that $R_n<r_0$, we have that Assumption \ref{ass: input space2} can be applied to $B\in\mathcal{B}_n$. Using Assumption \ref{ass:sampling}, we have that the random variable defined by $|\{i\in [\floor{\frac{n}{2}}] \}, s_i \in B\}|$ follows a binomial distribution with a success probability $p$ that can be lower bounded by $C_3'\frac{vol(B)}{vol(\inspace)} = C_3'' a_n^{\frac{d}{\alpha}}$ for $ C_3', C_3''>0$ and with expectation: 
$$
\mathbb{E} \left[|\{i\in [\floor{\frac{n}{2}}] \}, s_i \in B\}| \right] 
= \floor{\frac{n}{2}}p
\geq 
\frac{C_3''}{3} n^{\frac{\eta\alpha}{\eta\alpha + d}}\log(n)^{\frac{d}{d+\eta\alpha}} = C_3 n^{\frac{\eta\alpha}{\eta\alpha + d}}\log(n)^{\frac{d}{d+\eta\alpha}}
$$

where $C_3 := \frac{C_3''}{3}$. (Note: the "$3$" denominator was arbitrarily selected in order to remove the ceiling operator in the above equation).

As the only condition on $C_2$ is given by the bound $C_2> \frac{d + \alpha \xi}{\alpha \gamma(2 C_0)^\eta}$, $C_2$ can be set arbitrarily small as $C_0$ can be set arbitrarily large. Therefore, there exists $C_0>0$, $C_2>0$  such that $C_2 \leq \min\{\frac{d+ \eta \alpha}{ 2 C_1 \alpha },  \frac{C_3 (d + \eta \alpha)}{\alpha} \}$ which implies that $(\star \star \star)$ holds and that:
$$C_2 b_n = C_2 \left(\frac{\alpha}{d+\eta\alpha} \right)\log \left(\frac{n}{\log{(n)}}\right) \left(\frac{n}{\log(n)}\right)^{\frac{\eta\alpha}{\eta\alpha + d}} 
$$
$$\leq C_2 \left(\frac{\alpha}{d+\eta\alpha} \right) n^{\frac{\eta\alpha}{\eta\alpha + d}} \log(n)^{\frac{d}{d+\eta\alpha}}
\leq 
C_3 n^{\frac{\eta\alpha}{\eta\alpha + d}}\log(n)^{\frac{d}{d+\eta\alpha}} 
\leq
\frac{\mathbb{E}[|\{i\in [n] \}, s_i \in B\}|]}{2}.$$ 
This last relation implies that we can apply Lemma 1 of \cite{stone1982optimal} to obtain the upper bound:
$$
\pmeas \left(|\{i\in [n] \}, s_i \in B\}| \leq C_2 b_n \right)
\leq 
\pmeas \left(|\{i\in [n] \}, s_i \in B\}| \leq \frac{\mathbb{E} \left[|\{i\in [n] \}, s_i \in B\}| \right]}{2} \right)
$$
$$
\leq \left(\frac{2}{e} \right)^{\frac{\mathbb{E}[|\{i\in [n] \}, s_i \in B\}|]}{2}}
$$
which in turn implies
$$ 
\left(1 - \pmeas \left(|\{i\in [n] \}, s_i \in B\}| \leq C_2 b_n \right) \right)^{|\mathcal{B}_n|}
 \geq 
 \left(1 - \left(\frac{2}{e} \right)^{\frac{\mathbb{E}[|\{i\in [n] \}, s_i \in B\}|]}{2}} \right)^{|\mathcal{B}_n|} .
$$
Plugging this expression back into $\limsup_{n \to \infty} a_n^{-1}  \pmeas(S_n^c)$, we obtain
$$
a_n^{-1}  \pmeas \left(S_n^c \right)
 \leq 
 a_n^{-1} \left(1-\left(1 - (\frac{2}{e} )^{\frac{\mathbb{E}[|\{i\in [n] \}, s_i \in B\}|]}{2}}\right)^{|\mathcal{B}_n|}  \right).
$$
$$ \leq a_n^{-1} \left(1- \left(1 - (\frac{2}{e})^{\frac{C_3}{2} n^{\frac{\eta\alpha}{\eta\alpha + d}}\log(n)^{\frac{d}{d+\eta\alpha}}} \right)^{C_1{a_n}^{\frac{d}{\alpha}}}\right)
$$
which converges to $0$ as $n$ goes infinity and concludes the proof (follows from the exponential speed of convergence of $(\frac{2}{e})^{\frac{C_3}{2} n^{\frac{\eta\alpha}{\eta\alpha + d}}\log(n)^{\frac{d}{d+\eta\alpha}}}$).

($\star$) For completeness we revisit the assumption made in the proof. Recall for all $n\in\nat$,
$$E_n:= \left\{ \forall B \in \mathcal{B}_n, |\{i\in [n] \}, s_i \in B\}| > 0   \right\}.$$ 
Then, by law of total expectation, we have $\forall f \in Lip(L^*, \metric)$, $n \in \nat$ sufficiently large such that $\pmeas(E_n)>0$ (which exists since $\pmeas(E_n)>\pmeas(S_n)\overset{n\to\infty}{\rightarrow} 1$ );
$$ \mathbb{E}[a_n^{-1}\norm{\hat f_n - f}_\infty ]
$$
$$
=  a_n^{-1}(\mathbb{E} \left[\norm{\hat f_n - f}_\infty \Big| E_n \right]\pmeas(E_n)
+ \mathbb{E}\left[\norm{\hat f_n - f}_\infty \Big| E_n^c \right] \pmeas(E_n^c)).
$$

For all $ n \geq 1$, $f\in Lip(L^*, \metric)$ and any sampling procedure $\mathcal{D}_n$,  $\norm{\hat f_n - f}_\infty$
is uniformly bounded with probability $1$. More precisely, we have $\sup_{f\in Lip(L^*, \metric)}\norm{ \hat f_n - f}_\infty $ $\leq 2\obserrbnd + 2L\delta_{\metric} (\mathcal{ X})$ where $\delta_{\metric} (\mathcal{ X}):= \sup_{x,x'\in\mathcal{ X}}\metric(x,x')$ with probability $1$ which follows from $f\in Lip(L^*, \metric)$ and by the design of the Lipschitz interpolation framework. This bound is finite by the assumed compactness of $\inspace$. Therefore, denoting $K := 2\obserrbnd + 2L\delta_{\metric} (\mathcal{X})$, we have that the above statement is upper bounded by 
$$
\mathbb{E} \left[a_n^{-1}\norm{\hat f_n - f}_\infty \Big| E_n \right]
\pmeas(E_n) +  a_n^{-1} K \pmeas(E_n^c).
$$
The first term is equal to the simplified expression assumed earlier in the proof and the second term converges to 0 since $ \pmeas(E_n^c) \leq \pmeas(S_n^c)$ and $\lim_{n\to\infty} a_n^{-1} \pmeas(S_n^c) =0$ as shown above.

($\star\star$)
For all $B\in\mathcal{B}_n$, let $E(B)$ denote the event 
$$E(B):= \left\{\min_{s_i \in  B \cap  G_n^\inspace} e_i \in I_1, 
\max_{s_i \in B \cap  G_n^\inspace} e_i \in I_2 \right\}.$$
Then, imposing an arbitrary ordering of the hyperballs in $\mathcal{B}_n$, we have
$$\pmeas \left( \forall B \in \mathcal{B}_n,   \min_{s_i \in  B \cap  G_n^\inspace} e_i \in I_1, 
\max_{s_i \in B \cap  G_n^\inspace} e_i \in I_2 \right) 
$$
$$
=  \pmeas \left(E(B_1) \right)\prod_{i=2}^{|\mathcal{B}_n|}\pmeas \left(E(B_i) | E(B_{i-1}), ... , E(B_1) \right).
$$
For all $i \in \{1, ... , |\mathcal{B}_n| \}$, we observe that either 
there exists $j \in \{1,...,i-1\} \text{ such that }
B_i \cap B_j \cap  G_n^\inspace \neq 0$
in which case 
$$\pmeas \left(E(B_i) | E(B_{i-1}), ... , E(B_1) \right) > \pmeas(E(B_i))$$
or no such $j$ exists, in which case 
$$\pmeas(E(B_i) | E(B_{i-1}), ... , E(B_1)) = \pmeas(E(B_i)).$$
Therefore, we have that 
$$
\pmeas( \forall B \in \mathcal{B}_n,   \min_{s_i \in  B \cap  G_n^\inspace} e_i \in I_1, 
\max_{s_i \in B \cap  G_n^\inspace} e_i \in I_2) 
$$
$$
\geq \prod_{B\in\mathcal{B}_n}\pmeas( \min_{s_i \in  B \cap  G_n^\inspace} e_i \in I_1, 
\max_{s_i \in B \cap  G_n^\inspace} e_i \in I_2).
$$

$(\star \star \star)$ In order to alleviate notation, for all $B\in\mathcal{B}_n$, we define the following event:
$$\mathcal{E}_B(n) :=  \left\{|\{i\in [n] \}, s_i \in B\}| > C_2 b_n \right\}.$$
It is trivial to see that for all $B\in\mathcal{B}_n$, $\mathcal{E}_B(n)$ is increasing in $n$ (when $b_n$ is kept fixed). Utilising the arbitrary numbering of $\mathcal{B}_n$ defined above in $(\star \star)$, we have
$$ 
\pmeas(S_n) = \pmeas(\forall B \in \mathcal{B}_n:\mathcal{E}_B(n))
=  
\pmeas \left(\mathcal{E}_{B_2}(n) \right)\prod_{i=1}^{|\mathcal{B}_n|}\pmeas \left(\mathcal{E}_{B_i}(n)| \mathcal{E}_{B_{i-1}}(n), ... , \mathcal{E}_{B_1}(n) \right)
$$
$$
\geq \prod_{i=1}^{|\mathcal{B}_n|}\pmeas \left(\mathcal{E}_{B_i}(n - (i-1)C_2 b_n \right)
\geq 
\prod_{B \in \mathcal{B}_n}\pmeas \left(\mathcal{E}_{B}(n - |\mathcal{B}_n|C_2 b_n \right)
$$
where the second to last inequality holds due to the independence of the input sampling. Computing $|\mathcal{B}_n|C_2 b_n$, we obtain 
$$|\mathcal{B}_n|C_2 b_n \leq C_2  \frac{C_1}{a_n^{\frac{d}{\alpha}}} \log(\frac{1}{a_n})\frac{1}{a_n^\eta} = C_1 C_2 \frac{\log(\frac{1}{a_n})}{a_n^{\frac{d+\eta \alpha}{\alpha}}} $$
$$
= n C_1 C_2\frac{\alpha}{d + \eta \alpha}(1 - \frac{\log(\log(n))}{\log(n)})
\leq 
n C_1 C_2\frac{\alpha}{d + \eta \alpha}.
$$
Therefore, setting the condition $C_2 \leq \frac{d+ \eta \alpha}{ 2 C_1 \alpha }$, we have 
$$\pmeas(S_n) \geq \prod_{B \in \mathcal{B}_n}\pmeas \left(\mathcal{E}_{B} \left(n(1-C_1 C_2\frac{\alpha}{d + \eta \alpha}) \right) \right)
\geq 
\prod_{B \in \mathcal{B}_n} \pmeas \left(\mathcal{E}_{B}(\frac{n}{2}) \right).$$
\end{proof}

\section{Technical Lemmas}
\begin{lem} \label{lemma:tech conv equiv}
Assume that the settings and assumptions of Theorem \ref{thm:conv rate} hold. Let $(g_n)_{n \in \nat}$ denote a sequence of non-parametric predictors and $(b_n)_{n \in \nat}$ denote a convergence rate sequence (that converges to 0). If $\exists K > 0$ such that $\sup_{f\in \overline{Lip}(L^*, \metric, M^*)}\norm{\hat g_n - f}_\infty<K$ $\forall n \in \nat$ with probability 1, then
\begin{equation}
\lim_{n \to \infty} \sup_{f\in \overline{Lip}(L^*, \metric, M^*)} \mathbb{E}[(f(x_{n+1}) - \hat g_n(x_{n+1}))^p] \to 0 
\label{equ:conv moment}
\end{equation}
if and only if $\forall \epsilon>0$
\begin{equation}
\lim_{n \to \infty}\sup_{f \in \overline{Lip}(L^*, \metric, M^*)}\pmeas(|f(x_{n+1}) - \hat g_n(x_{n+1})| > \epsilon ) = 0
\label{equ:conv prob}
\end{equation}
where $\overline{Lip}(L^*,\metric, M^*)$ is as defined in Corollary \ref{cor:online}.
\end{lem}
\begin{proof}
"$\implies$" can be trivially obtained by applying Markov's inequality. We show the "$\impliedby$" statement.
Fix $\epsilon > 0$ and consider an arbitrary $f\in \overline{Lip}(L^*,\metric, M^*)$. Define $A_n : = |f(x_{n+1}) - \hat g_n(x_{n+1})| > \epsilon$,
we have
$$(f(x_{n+1}) - \hat g_n(x_{n+1}))^p \leq \epsilon^p 1_{A_n^c} + K^p 1_{A_n}$$
with probability $1$.
This implies that
$$\sup_{f\in \overline{Lip}(L^*, \metric, M^*)} \mathbb{E}[(f(x_{n+1}) - \hat g_n(x_{n+1}))^p]
$$
$$
\leq \epsilon^p  + K^p \sup_{f \in \overline{Lip}(L^*, \metric, M^*)} \pmeas(|f(x_{n+1}) - \hat g_n(x_{n+1})| > \epsilon)
$$
$$
\stackrel{n \to 0}{\leq} \epsilon^p.
$$
As the choice of $\epsilon$ was arbitrary, $(\ref{equ:conv moment})$ holds.
\end{proof}
\begin{lem} \label{lem:sequence proof}
$\forall p, c >0$, we have 
$$\limsup_{x \to \infty} x \left(1 - (1-\frac{1}{x^{p+1}})^{c x^p} \right) \leq 2 c $$
\end{lem}
\begin{proof}
Lemma \ref{lem:sequence proof} can be shown as follows.
$$ x \left(1 - (1-\frac{1}{x^{p+1}})^{c x^p} \right) = x \left(1-e^{c x^p\log(1-\frac{1}{x^{p+1}})} \right).$$
Expanding the exponent based on the power series expression of $\log(1+x)$, we obtain 
$$
c x^p\log(1-\frac{1}{x^{p+1}})
= - c x^p\sum_{m=1}^\infty \frac{1}{m x^{m(p+1)}}
=
- \frac{c x^p}{x^{p+1}} \sum_{m=0}^\infty \frac{1}{(m+1)x^{m(p+1)}}
$$
$$
\geq - \frac{c}{x} \sum_{m=0}^\infty \frac{1}{x^{m(p+1)}}  = - \frac{c}{x} \frac{x^{(p+1)}}{x^{(p+1)} - 1} \geq - \frac{2c}{x}
$$
for sufficiently large $x$. Substituting this equation back into the initial bound, we obtain:
$$\limsup_{x \to \infty} x \left(1-e^{c x^p\log(1-\frac{1}{x^{p+1}})} \right)
\leq \limsup_{x \to \infty}
x \left(1-e^{ - \frac{2c}{x}  } \right) \stackrel{x \to \infty}{\rightarrow} 2c.
$$
\end{proof}

\end{document}